\documentclass[a4,12pt]{amsart}

\usepackage{amssymb}
\usepackage{amstext}
\usepackage{amsmath}
\usepackage{amscd}
\usepackage{amsthm}
\usepackage{amsfonts}
\usepackage{enumerate}
\usepackage{graphicx}
\usepackage{latexsym}
\usepackage[usenames]{color}

\newtheorem*{maintheorem}{Main Theorem}
\newtheorem*{corollary*}{Corollary}
\newtheorem{theorem}{Theorem}[section]

\newtheorem{corollary}[theorem]{Corollary}
\newtheorem{lemma}[theorem]{Lemma}
\newtheorem{proposition}[theorem]{Proposition}

\newtheorem{claim}{Claim}
\newtheorem*{claim*}{Claim}

\theoremstyle{definition}
\newtheorem{definition}[theorem]{Definition}
\newtheorem{remark}[theorem]{Remark}

\newtheorem{notation}[theorem]{Notation}
\newtheorem{convention}[theorem]{Convention}

\theoremstyle{remark}

\renewcommand{\qedsymbol}{$\blacksquare$}

\numberwithin{equation}{theorem}

\def\Mod{\operatorname{\mathsf{Mod}}}
\def\mod{\operatorname{\mathsf{mod}}}

\newcommand{\CC}{{\mathcal C}}
\newcommand{\DD}{{\mathcal D}}

\newcommand{\UU}{{\mathcal U}}

\newcommand{\TT}{{\mathcal T}}

\newcommand{\XX}{{\mathcal X}}
\newcommand{\YY}{{\mathcal Y}}
\def\ZZ{\mathcal{Z}}

\newcommand{\Z}{\mathbb{Z}}

\newcommand{\add}{\operatorname{\mathsf{add}}}

\newcommand{\proj}{\operatorname{\rm proj}} 

\newcommand{\homo}[3]{{\rm{Hom}}_{#1}(#2, #3)} 
\newcommand{\ext}[4]{{\rm{Ext}}_{#1}^{#2}(#3, #4)} 
\newcommand{\End}{\operatorname{End}}
\newcommand{\gen}[3]{\langle #1{\rangle}_{#2}^{#3}}

\newcommand{\D}{\mathsf{D}}
\renewcommand{\AA}{{\mathcal A}}
\newcommand{\C}{\mathsf{C}}
\newcommand{\RHom}{\mathbf{R}\operatorname{Hom}\nolimits}
\def\Spec{\operatorname{Spec}}
\newcommand{\ctensor}{\operatorname{\widehat{\otimes}}} 
\newcommand{\tensor}{\operatorname{\otimes}}
\def\sCM{\operatorname{\underline{\mathsf{CM}}}}
\def\ess{\operatorname{\mathsf{ess}}}

\def\pd{\operatorname{pd}}
\def\Db{\mathsf{D}^{\mathrm{b}}}
\def\Dsg{\mathsf{D_{Sg}}}
\def\Dm{\mathsf{D}^{\mathrm{-}}}
\def\Dp{\mathsf{D}^{\mathrm{+}}}
\def\Da{\mathsf{D}^{\mathrm{\star}}}
\def\mdl{\operatorname{\mathsf{mod}}}
\def\p{\mathfrak{p}}
\def\q{\mathfrak{q}}
\def\m{\mathfrak{m}}
\def\n{\mathfrak{n}}
\def\rad{\operatorname{rad}}
\def\Min{\operatorname{Min}}
\def\sus{\mathsf{\Sigma}}
\def\gldim{\operatorname{gldim}}
\def\nil{\operatorname{nil}}
\def\LoL{\operatorname{\ell\ell}}
\def\codim{\operatorname{codim}}
\def\CI{\operatorname{CI}}
\def\Assh{\operatorname{Assh}}
\def\height{\operatorname{ht}}
\def\Max{\operatorname{Max}}
\def\Hom{\operatorname{Hom}}
\def\edim{\operatorname{edim}}
\def\Ext{\operatorname{Ext}}
\def\Tor{\operatorname{Tor}}
\def\cfrank{\operatorname{cf-rank}}

\makeatletter
\def\ddigit#1{\setbox0=\hbox{9}\setbox1=\hbox{#1}%
  \ifdim \wd0<\wd1 #1\else 0#1\fi}
\newcounter{hour}
\newcounter{HOUR}
\newcounter{minute}
 \divide \c@hour by 60
 \multiply \c@HOUR by 60
 \advance \c@minute by -\c@HOUR

\makeatother
\begin{document}
\setlength{\baselineskip}{15pt}
\title{Generators and dimensions of derived categories}
\author{Takuma Aihara}
\address{Division of Mathematical Science and Physics, Graduate School of Science and Technology, Chiba University, Yayoi-cho, Chiba 263-8522, Japan}
\email{taihara@math.s.chiba-u.ac.jp}
\author{Ryo Takahashi}
\address{Department of Mathematical Sciences, Faculty of Science, Shinshu University, 3-1-1 Asahi, Matsumoto, Nagano 390-8621, Japan}
\email{takahasi@math.shinshu-u.ac.jp}
\thanks{The second author was partially supported by JSPS Grant-in-Aid for Young Scientists (B) 22740008}
\begin{abstract}
Several years ago, Bondal, Rouquier and Van den Bergh introduced the notion of the dimension of a triangulated category, and Rouquier proved that the bounded derived category of coherent sheaves on a separated scheme of finite type over a perfect field has finite dimension.
In this paper, we study the dimension of the bounded derived category of finitely generated modules over a commutative Noetherian ring.
The main result of this paper asserts that it is finite over a complete local ring containing a field with perfect residue field.
Our methods also give a ring-theoretic proof of the affine case of Rouquier's theorem.
\end{abstract}
\maketitle
\tableofcontents
\section{Introduction}

The notion of the dimension of a triangulated category has been introduced by Bondal, Rouquier and Van den Bergh \cite{BV,R2}.
Roughly speaking, it measures how quickly the category can be built from a single object.
The dimensions of the bounded derived category of finitely generated modules over a Noetherian ring and that of coherent sheaves on a Noetherian scheme are called the {\em derived dimensions} of the ring and the scheme, while the dimension of the singularity category (in the sense of Orlov \cite{Or1}; the same as the stable derived category in the sense of Buchweitz \cite{Bu}) is called the {\em stable dimension}.
These dimensions have been in the spotlight in the studies of the dimensions of triangulated categories.

The importance of the notion of derived dimension was first recognized by Bondal and Van den Bergh \cite{BV} in relation to representability of functors.
They proved that smooth proper commutative/non-commutative varieties have finite derived dimension, which yields that every contravariant cohomological functor of finite type to vector spaces is representable.
The notion of stable dimension is closely related to that of representation dimension.
This invariant has been defined by Auslander \cite{A} to measure how far a given Artin algebra is from being representation-finite, and widely and deeply investigated so far;  see \cite{B,EHIS,I,O3,X} for example.
Rouquier \cite{R1} gave the first example of an Artin algebra of representation dimension more than three, by making use of the stable dimension of the exterior algebra of a vector space.

For small values of derived and stable dimensions, a number of definitive results have been obtained.
Rings of derived or stable dimension zero have been classified.
It is shown by Chen, Ye and Zhang \cite{CYZ} that a finite-dimensional algebra over an algebraically closed field has derived dimension zero if and only if it is an iterated tilted algebra of Dynkin type.
Yoshiwaki \cite{Yw} showed that a finite-dimensional self-injective algebra over an algebraically closed field has stable dimension zero if and only if it is representation-finite.
Recently, Minamoto \cite{M} established that the dimension of the derived category of perfect DG modules over a DG algebra is stable under separable algebraic extensions of the base field, which extends Yoshiwaki's result to algebras over perfect fields.
On the other hand, some geometric objects are known to have derived dimension one.
Orlov \cite{Or2} proved that a smooth quasi-projective curve has derived dimension one, and independently and in a different approach, Oppermann \cite{O5} showed that so do an elliptic curve and a weighted projective line of tubular type.

For general values of derived and stable dimensions, several lower and upper bounds have been found.
Bergh, Iyengar, Krause and Oppermann \cite{BIKO} proved that the stable dimension of a complete local complete intersection is at least its codimension minus one,  which implies a result of Oppermann \cite{O1} that gives a lower bound for the stable dimension of the group algebra of an elementary abelian group.
The result of Bergh, Iyengar, Krause and Oppermann was extended to an arbitrary commutative local ring by Avramov and Iyengar \cite{AI} by using the conormal free rank of the ring.
Rouquier \cite{R2} proved that the derived dimension of a reduced separated scheme of finite type over a field is not less than the dimension of the scheme.

As to upper bounds, the derived dimension of a ring is at most its Loewy length \cite{R2}.
In particular, Artinian rings have finite derived dimension.
Christensen, Krause and Kussin \cite{C,KK} showed that the derived dimension is bounded above by the global dimension, whence rings of finite global dimension are of finite derived dimension.
In relation to a conjecture of Orlov \cite{Or2}, a series of studies by Ballard, Favero and Katzarkov \cite{BF,BFK1,BFK2} gave in several cases upper bounds for derived and stable dimensions of schemes.
For instance, they obtained an upper bound of the stable dimension of an isolated hypersurface singularity by using the Loewy length of the Tjurina algebra.
On the other hand, there are a lot of triangulated categories having infinite dimension.
The dimension of the derived category of perfect complexes over a ring (respectively, a quasi-projective scheme) is infinite unless the ring has finite global dimension (respectively, the scheme is regular) \cite{R2}.
It has turned out by work of Oppermann and \v{S}\v{t}ov\'{i}\v{c}ek \cite{OS} that over a Noetherian algebra (respectively, a projective scheme) all proper thick subcategories of the bounded derived category of finitely generated modules (respectively, coherent sheaves) containing perfect complexes have infinite dimension.

As a main result of the paper \cite{R2}, Rouquier gave the following theorem.
\begin{theorem}[Rouquier]
Let $X$ be a separated scheme of finite type over a perfect field.
Then the bounded derived category of coherent sheaves on $X$ has finite dimension.
\end{theorem}
\noindent
Applying this theorem to an affine scheme, one obtains:
\begin{corollary}\label{rr}
Let $R$ be a commutative ring which is essentially of finite type over a perfect field $k$.
Then the bounded derived category $\Db(\mod R)$ of finitely generated $R$-modules has finite dimension, and so does the singularity category $\Dsg(R)$ of $R$.
\end{corollary}

The main purpose of this paper is to study the dimension and generators of the bounded derived category of finitely generated modules over a commutative Noetherian ring.
We will give lower bounds of the dimensions over general rings under some mild assumptions, and over some special rings we will also give upper bounds and explicit generators.
The main result of this paper is the following theorem.
(See Definition \ref{dtcat} for the notation.)

\begin{maintheorem}
Let $R$ be either a complete local ring containing a field with perfect residue field or a ring that is essentially of finite type over a perfect field.
Then there exist a finite number of prime ideals $\p_1,\dots,\p_n$ of $R$ and an integer $m\ge 1$ such that
$$
\Db(\mod R)=\gen{R/\p_1\oplus\cdots\oplus R/\p_n}{m}{}.
$$
In particular, $\Db(\mdl R)$ and $\Dsg(R)$ have finite dimension.
\end{maintheorem}

In Rouquier's result stated above, the essential role is played, in the affine case, by the Noetherian property of the enveloping algebra $R\tensor_kR$.
The result does not apply to a complete local ring, since it is in general far from being (essentially) of finite type and therefore the enveloping algebra is non-Noetherian.
Our methods not only show finiteness of dimensions over a complete local ring but also give a ring-theoretic proof of Corollary \ref{rr}.

The organization of this paper is as follows.

In Section \ref{prelim}, our basic definitions will be stated, including the definition of the dimension of a triangulated category.

In Section \ref{upper}, we will study generators of derived categories of an abelian category.
Using invariants of rings, we will obtain explicit upper bounds for the derived dimensions of quotient singularities, residue rings by monomials (e.g. Stanley-Reisner rings) and rings of Krull dimension one (e.g. numerical semigroup rings).

In Section \ref{finiteness}, we shall prove the main theorem stated above.
We will also provide generators of larger derived categories.

In Section \ref{app}, our main theorem will be applied to the stable category of Cohen-Macaulay modules.
We will observe that it has finite dimension over a complete Gorenstein local ring and a (not necessarily complete) Gorenstein excellent local ring with an isolated singularity.

In Section \ref{lower}, we will extend a result of Rouquier concerning lower bounds for the derived dimensions of affine algebras.
A similar lower bound will also be given for a general commutative Noetherian ring.
Both are described by using Krull dimension.
Finally, we will put a note on lower bounds of stable dimension.

\section{Preliminaries}\label{prelim}

This section is devoted to stating our convention, giving some basic notation and recalling the definition of the dimension of a triangulated category.

We assume the following throughout this paper.

\begin{convention}
\begin{enumerate}[(1)]
\item
All subcategories are full and closed under isomorphisms.
\item
All rings are associative and with identities.
\item
A Noetherian ring, an Artinian ring and a module mean a right Noetherian ring, a right Artinian ring and a right module, respectively.
\item
All complexes are cochain complexes.
\end{enumerate}
\end{convention}

We use the following notation.

\begin{notation}
\begin{enumerate}[(1)]
\item
Let $\AA$ be an abelian category.
\begin{enumerate}[(a)]
\item
We denote by $\proj\AA$ the subcategory of $\AA$ consisting of all projective objects of $\AA$.
When $\AA$ has enough projective objects, the {\em projective dimension} $\pd_{\AA}M$ of an object $M\in\AA$ is defined as the infimum of the integers $n$ such that there exists an exact sequence
$$
0 \to P_n \to P_{n-1} \to \cdots \to P_1 \to P_0 \to M \to 0
$$
in $\AA$ with $P_0,\dots,P_n\in\proj\AA$.
The \emph{global dimension} $\gldim\AA$ of $\AA$ is by definition the supremum of the projective dimensions of objects of $\AA$. 
\item
For a subcategory $\XX$ of $\AA$, the smallest subcategory of $\AA$ containing $\XX$ which is closed under finite direct sums and direct summands is denoted by $\add_\AA\XX$.
\item
We denote by $\C(\AA)$ the category of complexes of objects of $\AA$.
The derived category of $\AA$ is denoted by $\D(\AA)$.
The left bounded, the right bounded and the bounded derived categories of $\AA$ are denoted by $\Dp(\AA), \Dm(\AA)$ and $\Db(\AA)$, respectively.
We set $\D^\varnothing(\AA)=\D(\AA)$, and often write $\Da(\AA)$ with $\star\in\{\varnothing,+,-,\mathrm{b}\}$ to mean $\D^\varnothing(\AA)$, $\Dp(\AA)$, $\Dm(\AA)$ and $\Db(\AA)$.
\end{enumerate}
\item
Let $R$ be a ring. 
We denote by $\Mod R$ and $\mod R$ the category of $R$-modules and the category of finitely generated $R$-modules, respectively. 
For a subcategory $\XX$ of $\mod R$ (when $R$ is Noetherian), we put $\add_R\XX=\add_{\mod R}\XX$. 
\item
\begin{enumerate}[(a)]
\item
Let $\CC$ be an additive category, and let $\XX$ be a subcategory of $\CC$.
We say that $\XX$ is {\em closed under existing direct sums} provided that for any family $\{ X_\lambda\}_{\lambda\in\Lambda}$ of objects in $\XX$ the direct sum $\bigoplus_{\lambda\in\Lambda}X_\lambda$ belongs to $\XX$ whenever it exists in $\CC$.
\item
Let $\TT$ be a triangulated category and $\XX$ a subcategory of $\TT$.
We denote by $\ess_{\TT}\XX$ the smallest subcategory of $\TT$ containing $\XX$ which is closed under shifts and existing direct sums.
\end{enumerate}
\end{enumerate}
\end{notation}

\begin{remark}\label{iin}
Let $\AA$ be an abelian category, and let $\XX$ be a subcategory of $\AA$. 
Then
\begin{enumerate}[(1)]
\item
The subcategory $\ess_{\D(\AA)}\XX$ consists of all complexes of the form
$$
\bigoplus_{i\in\Z}\sus^{-i}X^i=(\cdots \xrightarrow{0} X^{i-1} \xrightarrow{0} X^i \xrightarrow{0} X^{i+1} \xrightarrow{0} \cdots)
$$
with $X^i\in\XX$ for all $i\in\Z$.
\item
We have $\ess_{\Da(\AA)}\XX=\ess_{\D(\AA)}\XX\cap\Da(\AA)$ for each $\star\in\{+,-,\mathrm{b}\}$.
\end{enumerate}
\end{remark}

The concept of the dimension of a triangulated category has been introduced by Rouquier \cite{R2}.
Now we recall its definition.

\begin{definition}\label{dtcat}
Let $\TT$ be a triangulated category.
\begin{enumerate}[(1)]
\item
A triangulated subcategory of $\TT$ is called {\em thick} if it is closed under direct summands.
\item
Let $\XX, \YY$ be two subcategories of $\TT$.
We denote by $\XX*\YY$ the subcategory of $\TT$ consisting of all objects $M$ that admit exact triangles
$$
X\to M\to Y\to \sus X
$$
with $X\in \XX$ and $Y\in \YY$.
We denote by $\gen{\XX}{}{}$ the smallest subcategory of $\TT$ containing $\XX$ which is closed under finite direct sums, direct summands and shifts.
For a non-negative integer $n$, we define the subcategory $\gen{\XX}{n}{}$ of $\TT$ by
\[\gen{\XX}{n}{}=\begin{cases}
\ \{0\} & (n=0), \\
\ \gen{\XX}{}{} & (n=1), \\
\ \gen{\gen{\XX}{}{}*\gen{\XX}{n-1}{}}{}{} & (2\le n<\infty).
\end{cases}\] 
Put $\gen{\XX}{\infty}{}=\bigcup_{n\ge0}\gen{\XX}{n}{}$.
When the ground category $\TT$ should be specified, we write $\gen{\XX}{n}{\TT}$ instead of $\gen{\XX}{n}{}$.
For a ring $R$ and a subcategory $\XX$ of $\D(\Mod R)$, we put $\gen{\XX}{n}{R}=\gen{\XX}{n}{\D(\Mod R)}$.
\item
The \emph{dimension} of $\TT$, denoted by $\dim\TT$, is the infimum of the integers $d$ such that there exists an object $M\in\TT$ with $\gen{M}{d+1}{}=\TT$.
\end{enumerate}
\end{definition}

\begin{remark}\label{hee}
\begin{enumerate}[(1)]
\item
Let $\TT$ be a triangulated category.
\begin{enumerate}[(a)]
\item
The subcategory $\gen{\XX}{\infty}{}$ of $\TT$ is nothing but the smallest thick subcategory of $\TT$ containing $\XX$.
\item
The octahedral axiom implies that $*$ satisfies associativity: $(\XX*\YY)*\ZZ=\XX*(\YY*\ZZ)$ for all subcategories $\XX,\YY,\ZZ$ of $\TT$.
\item
Let $\UU$ be a thick subcategory of $\TT$.
Then $\gen{\XX}{n}{\UU}=\gen{\XX}{n}{\TT}$ for any subcategory $\XX$ of $\UU$ and any integer $n\ge0$.
\item
The dimension of $\TT$ can be infinite: $\dim\TT=\infty$ if and only if $\gen{M}{n}{}\ne\TT$ for every object $M\in\TT$ and every integer $n\ge0$.
\end{enumerate}
\item
Let $\AA$ be an abelian category, and let $\XX$ be a subcategory of $\AA$ that is closed under existing direct sums.
Then $\gen{\ess_{\Db(\AA)}\XX}{n}{\Db(\AA)}=\gen{\XX}{n}{\Db(\AA)}$ for $n\ge0$.
\end{enumerate}
\end{remark}

\section{Upper bounds}\label{upper}

The aim of this section is to find explicit generators and upper bounds of dimensions for derived categories in several cases.
Our first goal is to do this in the case of finite global dimension.
It requires us to make the following easy observation, which will also be used in a later section.

\begin{lemma}\label{split out}
Let $\AA$ be an abelian category.
Let $X$ be a complex of objects of $\AA$.
Suppose that the homology ${\rm H}^iX$ is a projective object of $\AA$ for every integer $i$.
Then there exists a quasi-isomorphism $\bigoplus_{i\in\Z}\sus^{-i}{\rm H}^iX\to X$ in $\C(\AA)$.
Hence one has an isomorphism $X\cong\bigoplus_{i\in\Z}\sus^{-i}{\rm H}^iX$ in $\D(\AA)$. 
\end{lemma}

\begin{proof}
Fix an integer $i$. 
Since ${\rm H}^iX$ is projective, the natural surjection ${\rm Z}^iX\to {\rm H}^iX$ is a split epimorphism, and there is a split monomorphism $f^i:{\rm H}^iX\to {\rm Z}^iX$.
Let $g^i:\mathrm{H}^i{X}\to X^i$ be the composition of $f^i$ and the inclusion map $\mathrm{Z}^iX\to X^i$.
We have a morphism
$$
\begin{CD}
\bigoplus_{i\in\Z}\sus^{-i}{\rm H}^iX @. = (\cdots @>0>> {\rm H}^{i-1}X @>0>> {\rm H}^iX @>0>> {\rm H}^{i+1}X @>0>> \cdots) \\
@V{g}VV @. @V{g^{i-1}}VV @V{g^i}VV @V{g^{i-1}}VV \\
X @. = (\cdots @>>> X^{i-1} @>>> X^i @>>> X^{i+1} @>>> \cdots)
\end{CD}
$$
in $\C(\AA)$.
It is evident that $g$ is a quasi-isomorphism.
\end{proof}

The first goal in this section is to extend a result of Krause and Kussin \cite[Lemma 2.5]{KK} on the bounded derived category of an abelian category to larger derived categories.

\begin{proposition}\label{KKp}
Let $\AA$ be an abelian category with enough projective objects, and let $n\geq0$ be an integer. 
Then for each $\star\in\{\varnothing,+,-,\mathrm{b}\}$ there is an inclusion 
$$
\{\,X\in\Da(\AA)\mid\pd_{\AA}{\rm H}^iX\le n\text{ for all }i\,\}\subseteq\gen{\ess_{\Da(\AA)}(\proj\AA)}{n+1}{\Da(\AA)}.
$$
In particular, every object $X\in\Db(\AA)$ with $\pd_{\AA}{\rm H}^iX\le n$ for all $i$ belongs to $\gen{\proj\AA}{n+1}{\Db(\AA)}$.
\end{proposition}

Krause and Kussin proved the last assertion of the proposition, assuming that each boundary $\mathrm{B}^iX$ also has projective dimension at most $n$.
We do not need this assumption, while the proof is similar to theirs.

\begin{proof}
The last assertion follows from the first one and Remark \ref{hee}(2).

Let $X\in\Da(\AA)$ with $\pd_\AA\mathrm{H}^iX\le n$ for all $i$.
Fix an integer $i$.
We have a commutative diagram
\[\begin{CD}
0 @>>> P_{{\rm B}^iX} @>>> P_{{\rm Z}^iX} @>>> P_{{\rm H}^iX} @>>> 0 \\
@. @VVV @VVV @V{g^i}VV \\
0 @>>> {\rm B}^iX @>>> {\rm Z}^iX @>>> {\rm H}^iX @>>> 0
\end{CD}\]
in $\AA$ with exact rows and surjective columns such that $P_{{\rm B}^iX},P_{{\rm H}^iX}$ are projective, that the kernel of $g^i$ has projective dimension at most $n-1$ and that $P_{{\rm Z}^iX}=P_{{\rm B}^iX}\oplus P_{{\rm H}^iX}$.
Similarly, there is a commutative diagram
$$
\begin{CD}
0 @>>> P_{{\rm Z}^iX} @>>> P_{X^i} @>>> P_{{\rm B}^{i+1}X} @>>> 0 \\
@. @VVV @V{f^i}VV @VVV \\
0 @>>> {\rm Z}^iX @>>> X^i @>>> {\rm B}^{i+1}X @>>> 0
\end{CD}
$$
in $\AA$ with exact rows and surjective columns, where $P_{X^i}=P_{{\rm Z}^iX}\oplus P_{{\rm B}^{i+1}X}$.
Put $P^i=P_{X^i}$ and let $\partial^i$ be the composite map
$$
P^{i}=P_{X^i}\to P_{{\rm B}^{i+1}X}\to P_{{\rm Z}^{i+1}X}\to P_{X^{i+1}}=P^{i+1}.
$$
Then we have $\partial^{i+1}\partial^i=0$, and get a complex $P=(\cdots\xrightarrow{\partial^{i-1}} P^i\xrightarrow{\partial^i} P^{i+1}\xrightarrow{\partial^{i+1}}\cdots)$.
Note that $P$ is in $\Da(\AA)$.
It is easy to see that the maps $f^i:P^i\to X^i$ form a morphism $f:P\to X$ in $\C(\AA)$.
For any $i\in\Z$ we have ${\rm H}^iP=P_{{\rm H}^iX}$, which is projective.
By Lemma \ref{split out}, $P$ is isomorphic to $\bigoplus_{i\in\Z}\sus^{-i}P_{{\rm H}^iX}$ in $\D(\AA)$, which implies that $P$ belongs to $\gen{\ess_{\Da(\AA)}(\proj\AA)}{}{}$; see Remark \ref{iin}.
Now we take an exact triangle
$$
P\xrightarrow{f}X\to X_1\to\sus P
$$
in $\D(\AA)$.
Note that for each $i$ the map ${\rm H}^if$ coincides with $g^i$, which is surjective. 
We have an exact sequence
$$
\cdots\to{\rm H}^iP\xrightarrow{{\rm H}^if} {\rm H}^iX\xrightarrow{0}{\rm H}^iX_1\to {\rm H}^{i+1}P\xrightarrow{{\rm H}^{i+1}f} {\rm H}^{i+1}X\xrightarrow{0}\cdots,
$$
which gives rise to a short exact sequence
$$
0 \to \mathrm{H}^iX_1 \to P_{\mathrm{H}^{i+1}X} \xrightarrow{g^{i+1}} \mathrm{H}^{i+1}X \to 0.
$$
Hence ${\rm H}^iX_1$ has projective dimension at most $n-1$ for all $i\in\Z$. 
Put $P_0=P$, $X_0=X$ and $f_0=f$. 
Inductively, we obtain a diagram
$$
\begin{CD}
P_0 @. P_1 @. P_2 @. \cdots @. P_n \\
@V{f_0}VV @V{f_1}VV @V{f_2}VV @. @| \\
X_0 @>>> X_1 @>>> X_2 @>>> \cdots @>>> X_n.
\end{CD}
$$
in $\D(\AA)$ with $P_j\in \gen{\ess_{\Da(\AA)}(\proj\AA)}{}{}$ such that there is an exact triangle $P_j\xrightarrow{f_j} X_j\to X_{j+1}\to \sus P_j$ for $0\leq j\leq n$. 
Thus $X$ belongs to $\gen{\ess_{\Da(\AA)}(\proj\AA)}{n+1}{}$.
\end{proof}

The following is immediate from Proposition \ref{KKp}.
The third assertion is a categorical version of \cite[Proposition 2.6]{KK}.

\begin{corollary}\label{krskssn}
Let $\AA$ be an abelian category with enough projective objects and of finite global dimension.
\begin{enumerate}[\rm(1)]
\item
For each $\star\in\{\varnothing,+,-\}$ one has $\Da(\AA)=\gen{\ess_{\Da(\AA)}(\proj\AA)}{\gldim\AA+1}{}$.
\item
It holds that $\Db(\AA)=\gen{\proj\AA}{\gldim\AA+1}{}$.
\item
If $\proj\AA=\add_\AA G$ for some $G\in\AA$, then $\dim\Db(\AA)\le\gldim\AA$.
\end{enumerate}
\end{corollary}

\begin{remark}
Under the assumption of the third assertion in the above corollary, $\AA$ is equivalent to $\mod\End_\AA(G)$. 
Indeed, the proof of \cite[Proposition 3.4]{KK} shows that $\End_\AA(G)$ is right coherent, that is, $\mod\End_\AA(G)$ is an abelian category.
The exact functor $\Hom_\AA(G, -):\AA\to \mod\End_\AA(G)$ is an equivalence.
Here we do not need the assumption that $\AA$ has finite global dimension. 
\end{remark}

Next we will apply Corollary \ref{krskssn} to module categories of Noetherian rings.
We say that an exact functor $F:\TT\to \TT'$ of triangulated categories is \emph{essentially dense} if for every object $X'\in\TT'$ there exists an object $X\in\TT$ such that $X'$ is isomorphic to a direct summand of $FX$.
We will need the lemma below, which says that an essentially dense functor does not increase the dimension.

\begin{lemma}\label{essentially dense}
Let $F:\TT\to \TT'$ be an essentially dense exact functor of triangulated categories.
Let $\XX$ be a subcategory of $\TT$ and $n\ge0$ an integer. 
If $\TT=\gen{\XX}{n}{}$ holds, then $\TT'=\gen{F\XX}{n}{}$ holds. 
In particular, there is an inequality $\dim\TT'\leq\dim\TT$.
\end{lemma}

\begin{proof}
Take an object $T'\in\TT'$.
Since $F$ is essentially dense, we can choose an object $T\in\TT$ such that $T'$ is isomorphic to a direct summand of $FT$. 
The object $T$ belongs to $\gen{\XX}{n}{}$.
Since $F$ is exact, $FT$ is in $\gen{F\XX}{n}{}$, and so is $T'$.
Hence $\TT'=\gen{F\XX}{n}{}$.
\end{proof}

For rings $R$ such that there exists an $R$-algebra $S$ of finite global dimension containing $R$ as a direct summand, one can give generators of the derived categories.

\begin{proposition}\label{dirsmd}
Let $\phi:R\to S$ be a ring homomorphism.
Assume that $\phi$ is a split monomorphism of $(R,R)$-bimodules and that $S$ has global dimension $n$.
Then the following hold.
\begin{enumerate}[\rm (1)]
\item
$\Dm(\Mod R)=\gen{\ess_{\Dm(\Mod R)}S}{n+1}{}$.
\item
If $R,S$ are right Noetherian and $S$ is finitely generated as a right $R$-module, then:
\begin{enumerate}[\rm (a)]
\item
$\Dm(\mdl R)=\gen{\ess_{\Dm(\mod R)}S}{n+1}{}$.
\item
$\Db(\mdl R)=\gen{S}{n+1}{}$.
In particular, the inequality $\dim\Db(\mdl R)\le n=\gldim S$ holds.
\end{enumerate}
\end{enumerate}
\end{proposition}

\begin{proof}
(1) Corollary \ref{krskssn}(1) yields $\Dm(\Mod S)=\gen{\ess_{\Dm(\Mod S)}(\proj(\Mod S))}{n+1}{\Dm(\Mod S)}=\gen{\ess_{\Dm(\Mod S)}S}{n+1}{S}$. 
Since $\phi$ is a split monomorphism of $(R,R)$-bimodules, every object $X\in\Dm(\Mod R)$ is isomorphic to a direct summand of $X\tensor_R^{\bf L}S_R=(X\tensor_R^{\bf L}S)\tensor_S^{\bf L}S_R$.
Hence the functor $-\tensor_S^{\bf L}S_R:\Dm(\Mod S)\to \Dm(\Mod R)$ is essentially dense. 
By Lemma \ref{essentially dense}, we obtain $\Dm(\Mod R)=\gen{\ess_{\Dm(\Mod R)}S}{n+1}{R}$. 

(2) As $R$ and $S$ are right Noetherian, $\mod R$ and $\mod S$ are abelian categories with enough projective objects. 
Since $S$ is a finitely generated right $R$-module, one has an exact functor $-\tensor^{\bf L}_SS_R:\Dm(\mod S)\to \Dm(\mod R)$, which is essentially dense by the proof of (1). 
Let $X$ be an object in $\Db(\mod R)$.
Since $R$ is isomorphic to a direct summand of $S$ as an $(R, R)$-bimodule, $X$ is isomorphic to a direct summand of $X\tensor_RS$ as a bounded complex of right $R$-modules.
We can naturally regard $X\tensor_RS$ as a bounded complex of right $S$-modules. 
We have isomorphisms $(X\tensor_RS)\tensor^{\bf L}_{S}S_R \cong (X\tensor_RS)\tensor_SS_R\cong X\tensor_RS_R$ in $\Db(\mod R)$.
Thus $-\tensor^{\bf L}_{S}S_R:\Db(\mod S)\to \Db(\mod R)$ is an essentially dense functor.
Now the assertions (a) and (b) follow from Corollary \ref{krskssn}(1)(2) and Lemma \ref{essentially dense}.
\end{proof}

Applying the above proposition, we observe that the dimensions of the bounded derived categories of finitely generated modules over quotient singularities are at most their (Krull) dimensions, particularly that they are finite.

\begin{corollary}
Let $S$ be either the polynomial ring $k[x_1,\dots,x_n]$ or the formal power series ring $k[[x_1,\dots,x_n]]$ over a field $k$.
Let $G$ be a finite subgroup of the general linear group $\operatorname{GL}_n(k)$, and assume that the characteristic of $k$ does not divide the order of $G$.
Let $R=S^G$ be the invariant subring.
Then $\Db(\mdl R)=\gen{S}{n+1}{}$ holds, and hence one has
$$
\dim\Db(\mdl R)\le n=\dim R<\infty.
$$
\end{corollary}

\begin{proof}
Let $S=k[x_1,\dots,x_n]$.
As the order of $G$ is invertible in $k$, a Raynolds operator for $(R, S)$ can be defined, and $R$ is a direct summand of $S$ as an $R$-module; see \cite[Page 281]{BH}.
Since $S$ is a finitely generated $R$-algebra and an integral extension of $R$ by \cite[Page 282]{BH}, it is a finitely generated $R$-module.
Moreover, $S$ has global dimension $n$.
Thus the assertion follows from Proposition \ref{dirsmd}(2)(b).
The case where $S=k[[x_1,\dots,x_n]]$ is similar.
\end{proof}

On the other hand, taking an algebra that is a finitely generated free module does not decrease the dimension.
Compare the result below with Proposition \ref{dirsmd}(2)(b).

\begin{proposition}
Let $R\to S$ be a homomorphism of Noetherian rings.
Assume that $S$ is a finitely generated free $R$-module.
If $\Db(\mdl S)=\gen{G}{n}{S}$ for some object $G$ and some integer $n$, then $\Db(\mdl R)=\gen{G}{n}{R}$.
Therefore $\dim\Db(\mdl R)\le\dim\Db(\mdl S)$.
\end{proposition}

\begin{proof}
Let $X$ be an object of $\Db(\mod R)$.
We have $X\tensor^{\bf L}_{R}S\cong X\tensor_RS\in\Db(\mod S)=\gen{G}{n}{S}$.
Applying the exact functor $-\tensor_S^{\bf L}S_R:\Db(\mod S)\to\Db(\mod R)$ implies that $X\tensor^{\bf L}_{R}S_R$ belongs to $\gen{G}{n}{R}$.
Since there are isomorphisms $X\tensor^{\bf L}_{R}S_R\cong X\tensor_RS_R\cong X^{\oplus r}$ in $\Db(\mod R)$, where $r$ is the free rank of $S$ over $R$, the object $X^{\oplus r}$ is in $\gen{G}{n}{R}$, and so is its direct summand $X$.
\end{proof}

For a commutative ring $R$, we denote the set of minimal prime ideals of $R$ by $\Min R$.
As is well-known, $\Min R$ is a finite set whenever $R$ is Noetherian. 
Also, we denote by $\lambda(R)$ the infimum of the integers $n\ge0$ such that there is a filtration
$$
0=I_0\subseteq I_1\subseteq \cdots\subseteq I_n=R
$$
of ideals of $R$ with $I_i/I_{i-1}\cong R/\p_i$ for $1\le i\le n$, where $\p_i\in \Spec R$.
If $R$ is Noetherian, then such a filtration exists and $\lambda(R)$ is a non-negative integer.

\begin{proposition}\label{minprime}
Let $R$ be a Noetherian commutative ring.
\begin{enumerate}[\rm(1)]
\item
Suppose that for every $\p\in\Min R$ there exist an $R/\p$-complex $G(\p)$ and an integer $n(\p)\ge0$ such that $\Db(\mdl R/\p)=\gen{G(\p)}{n(\p)}{}$.
Then $\Db(\mdl R)=\gen{G}{n}{}$ holds, where $G=\bigoplus_{\p\in\Min R}G(\p)$ and $n=\lambda(R)\cdot\max\{\,n(\p)\mid\p\in\Min R\,\}$.
\item
There is an inequality
$$
\dim \Db(\mdl R)\le \lambda(R)\cdot\sup\{\,\dim\Db(\mod R/\p)+1\mid\p\in\Min R\,\}-1.
$$
\end{enumerate}
\end{proposition}

\begin{proof}
(1) There exists a filtration
$$
0=I_0\subsetneq I_1\subsetneq\cdots\subsetneq I_{\lambda(R)}=R
$$
of ideals of $R$ such that for each $1\le i\le\lambda(R)$ one has $I_i/I_{i-1}\cong R/\p_i$, where $\p_i$ is a prime ideal of $R$.
Choose a minimal prime ideal $\q_i$ contained in $\p_i$.
Let $X$ be an object of $\Db(\mdl R)$.
There is a filtration
$$
0=XI_0\subseteq XI_1\subseteq\cdots\subseteq XI_{\lambda(R)}=X
$$
of subcomplexes of the $R$-complex $X$.
Note that the residue complex $XI_i/XI_{i-1}$ can be regarded as an object of $\Db(\mdl R/\q_i)$.
Since $\Db(\mdl R/\q_i)=\gen{G(\q_i)}{n(\q_i)}{R/\q_i}$, we have that $XI_i/XI_{i-1}$, as an object of $\Db(\mdl R)$, belongs to the subcategory $\gen{G(\q_i)}{n(\q_i)}{R}$ of $\Db(\mdl R)$.
The exact triangles
$$
XI_{i-1} \to XI_i \to XI_i/XI_{i-1} \to \sus XI_{i-1}
$$
show that $X$ is in $\gen{\bigoplus_{i=1}^{\lambda(R)}G(\q_i)}{\sum_{i=1}^{\lambda(R)}n(\q_i)}{}$. 
Since $\bigoplus_{i=1}^{\lambda(R)}G(\q_i)\in\gen{G}{}{}$ and $\sum_{i=1}^{\lambda(R)}n(\q_i)\le n$, the assertion holds. 

(2) This immediately follows from (1).
\end{proof}

\begin{remark}
With the notation of the proof of Proposition \ref{minprime}, we easily see that $\Spec R=V(\p_1)\cup\cdots\cup V(\p_m)=V(\q_1)\cup\cdots\cup V(\q_m)$.
Hence we have $\Min R=\{\q_1,\dots,\q_m\}$.
\end{remark}

The proposition above yields the result below, which implies for example that the bounded derived categories of finitely generated modules over Stanley-Reisner rings have finite dimension.

\begin{corollary}
Let $k$ be a field.
Let $R$ be either $k[x_1,\dots,x_r]/(m_1,\dots,m_s)$ or $k[[x_1,\dots,x_r]]/(m_1,\dots,m_s)$, where $m_1,\dots,m_s$ are monomials.
Then $\Db(\mdl R)=\gen{G}{n}{}$, where $G=\bigoplus_{\p\in\Min R}R/\p$ and $n=\lambda(R)\cdot\max\{\,\dim R/\p+1\mid\p\in\Min R\,\}$. 
In particular, one has
$$
\dim \Db(\mdl R)\le \lambda(R)\cdot\max\{\,\dim R/\p+1\mid\p\in\Min R\,\}-1<\infty.
$$
\end{corollary}

\begin{proof}
Note that each minimal prime ideal $\p$ of $R$ is generated by some of the variables $x_1,\dots,x_r$.
Hence the residue ring $R/\p$ is isomorphic to a polynomial ring or a formal power series ring over $k$, and the global dimension of $R/\p$ is equal to $\dim R/\p$.
Now the corollary follows from Proposition \ref{minprime} and Corollary \ref{krskssn}(2).
\end{proof}

Let $R$ be a commutative Noetherian ring.
We set
\begin{align*}
\LoL(R) & =\inf\{\,n\ge0\mid(\rad R)^n=0\,\}, \\
r(R) & =\min\{\,n\ge0\mid(\nil R)^n=0\,\},
\end{align*}
where $\rad R$ and $\nil R$ denote the Jacobson radical and the nilradical of $R$, respectively.
The first number is called the {\em Loewy length} of $R$ and is finite if (and only if) $R$ is Artinian, while the second one is always finite.
Let $R_{\rm red}=R/\nil R$ be the associated reduced ring.
When $R$ is reduced, we denote by $\overline R$ the integral closure of $R$ in the total quotient ring $Q$ of $R$.
Let $C_R$ denote the {\em conductor} of $R$, i.e., $C_R$ is the set of elements $x\in Q$ with $x\overline R\subseteq R$.
We can give an explicit generator and an upper bound of the dimension of the bounded derived category of finitely generated modules over a one-dimensional complete local ring.

\begin{proposition}\label{dimension one}
Let $R$ be a Noetherian commutative complete local ring of Krull dimension one with residue field $k$.
Then it holds that $\Db(\mdl R)=\gen{\overline{R_{\rm red}}\oplus k}{r(R)\cdot(2\LoL(R_{\rm red}/C_{R_{\rm red}})+2)}{}$.
In particular,
$$
\dim\Db(\mdl R)\le r(R)\cdot(2\LoL(R_{\rm red}/C_{R_{\rm red}})+2)-1<\infty.
$$
\end{proposition}

\begin{proof}
By \cite[Lemma 7.35]{R2}, we may assume that $R=R_{\rm red}$.
Then we have $r(R)=1$.
The ring $\overline R$ is finitely generated as an $R$-module; see \cite[Corollary 4.6.2]{SH}.
We have that $\overline R$ is a Noetherian normal ring of Krull dimension one, and hence it is regular.
Therefore, $\overline R$ has global dimension one, which implies $\Db(\mdl\overline R)=\gen{\overline R}{2}{}$ by Corollary \ref{krskssn}(2). 
On the other hand, the residue ring $R/C_R$ is Artinian by \cite[Page 236]{SH}, so $\Db(\mdl R/C_R)=\gen{k}{\LoL(R/C_R)}{}$ holds by \cite[Proposition 7.37]{R2}.
Now, let $X$ be an object of $\Db(\mdl R)$, i.e., a bounded complex of finitely generated $R$-modules.
Applying $X\otimes_R-$ to the natural exact sequence $0\to R\xrightarrow{f} \overline R\to \overline R/R\to0$ of $(R, R)$-bimodules, we get an exact sequence 
$$
0\to K\to X\xrightarrow{X\tensor_Rf} X\tensor_R\overline R\to X\tensor_R\overline R/R\to0
$$
in $\C^{\rm b}(\mod R)$. 
Decomposing this into two short exact sequences, we obtain two exact triangles
\begin{align*}
& K\to X\to Y\to \sus K,\\
& Y\to X\tensor_R\overline R\to X\tensor_R\overline R/R\to \sus Y
\end{align*}
in $\Db(\mod R)$.
Since $C_R$ is a common ideal of $R$ and $\overline R$ by \cite[Exercise 2.11]{SH}, $\overline{R}/R$ is annihilated by $C_R$, and hence $\overline{R}/R$ is an $(R,R/C_R)$-bimodule.
The complexes $X\tensor_R\overline R/R$ and $X\tensor_R\overline R$ can be regarded as objects of 
$\Db(\mod R/C_R)$ and $\Db(\mod \overline R)$, respectively. 
Thus we see that $Y$ belongs to $\gen{\overline R\oplus k}{\LoL(R/C_R)+2}{}$. 
As each $K^i$ is a homomorphic image of $\Tor_1^R(X^i,\overline{R}/R)$, it is annihilated by $C_R$.
Therefore the complex $K$ can be regarded as an object of $\Db(\mod R/C_R)$. 
Consequently, the object $X$ of $\Db(\mod R)$ belongs to $\gen{\overline R\oplus k}{2\LoL(R/C_R)+2}{}$. 
\end{proof}

Let $R$ be a commutative Noetherian local ring of Krull dimension $d$ with maximal ideal $\m$. 
We denote by $e(R)$ the multiplicity of $R$, that is, $e(R)=\lim_{n\to\infty}\frac{d!}{n^d}\ell_R(R/\m^{n+1})$.
Recall that a {\em numerical semigroup} is defined as a subsemigroup $H$ of the additive semigroup $\mathbb{N}=\{0,1,2,\dots\}$ containing $0$ 
such that $\mathbb{N}\backslash H$ is a finite set. 
For a numerical semigroup $H$, let $c(H)$ denote the {\em conductor} of $H$, that is, 
$$
c(H)=\max\{\,i\in\mathbb{N}\mid i-1\notin H\,\}.
$$
For a real number $\alpha$, put $\lceil\alpha\rceil=\min\{\,n\in\Z\mid n\ge\alpha\,\}$.
Making use of the above proposition, one can get an upper bound of the dimension of the bounded derived category of finitely generated modules over a {\em numerical semigroup ring} $k[[H]]$ over a field $k$, in terms of the conductor of the semigroup and the multiplicity of the ring.

\begin{corollary}\label{mtk}
Let $k$ be a field and $H$ be a numerical semigroup.
Let $R$ be the numerical semigroup ring $k[[H]]$, that is, the subring $k[[t^h|h\in H]]$ of $S=k[[t]]$.
Then $\Db(\mdl R)=\gen{S\oplus k}{2\left\lceil\frac{c(H)}{e(R)}\right\rceil+2}{}$ holds.
Hence
$$
\dim\Db(\mdl R)\le2\left\lceil\frac{c(H)}{e(R)}\right\rceil+1.
$$
\end{corollary}

\begin{proof}
One can write $H={\langle a_1,\dots,a_n\rangle}$ with $0<a_1<\cdots<a_n$ and $\gcd(a_1,\dots,a_n)=1$ (cf. \cite[Page 178]{BH}).
Then one has $R=k[[t^{a_1},\dots,t^{a_n}]]$ and observes that $e(R)=a_1$.
As $R$ is a domain, $R_{\rm red}=R$ and $r(R)=1$ hold.
Note that $\overline R=S$.
Putting $c=c(H)$, one has $C_R=t^cS$.
The equalities
$$
\LoL(R/C_R)=\LoL(R/t^cS)=\min\{\,i\mid\m^i\subseteq t^cS\,\}=\min\{\,i\mid a_1i\ge c\,\}=\left\lceil\frac{c}{a_1}\right\rceil.
$$
hold, where $\m=(t^{a_1},\dots,t^{a_n})$ is the maximal ideal of $R$.
Proposition \ref{dimension one} completes the proof.
\end{proof}

\section{Finiteness}\label{finiteness}

In this section, we consider finiteness of the dimension of the bounded derived category of finitely generated modules over a complete local ring. 
Let $R$ be a commutative algebra over a field $k$.
Rouquier \cite{R2} proved the finiteness of the dimension of $\Db(\mod R)$ when $R$ is an affine $k$-algebra, where the fact that the enveloping algebra $R\tensor_kR$ is Noetherian played a crucial role.
The problem in the case where $R$ is a complete local ring is that one cannot hope that $R\tensor_kR$ is Noetherian.
Our methods instead use the completion of the enveloping algebra, that is, the complete tensor product $R\ctensor_kR$, which is a Noetherian ring whenever $R$ is a complete local ring with coefficient field $k$.

Let $R$ and $S$ be commutative Noetherian complete local rings with maximal ideals $\m$ and $\n$, respectively.
Suppose that they contain fields and have the same residue field $k$, i.e., $R/\m\cong k\cong S/\n$.
Then Cohen's structure theorem yields isomorphisms
\begin{align*}
R & \cong k[[x_1,\dots,x_m]]/(f_1,\dots,f_a),\\
S & \cong k[[y_1,\dots,y_n]]/(g_1,\dots,g_b).
\end{align*}
We denote by $R\ctensor_kS$ the {\em complete tensor product} of $R$ and $S$ over $k$, namely,
$$
R\ctensor_kS=\varprojlim_{i,j}(R/\m^i\tensor_kS/\n^j).
$$
For $r\in R$ and $s\in S$, we denote by $r\ctensor s$ the image of $r\tensor s$ by the canonical ring homomorphism $R\tensor_kS\to R\ctensor_kS$. 
Note that there is a natural isomorphism
$$
R\ctensor_kS\cong k[[x_1,\dots,x_m,y_1,\dots,y_n]]/(f_1,\dots,f_a,g_1,\dots,g_b).
$$
Details of complete tensor products can be found in \cite[Chapter V]{S}.

Recall that a ring extension $A\subseteq B$ is called {\em separable} if $B$ is projective as a $B\tensor_AB$-module.
This is equivalent to saying that the map $B\tensor_AB\to B$ given by $x\tensor y\mapsto xy$ is a split epimorphism of $B\tensor_AB$-modules.

Now, let us prove our main theorem. 

\begin{theorem}\label{main result}
Let $R$ be a Noetherian complete local commutative ring containing a field with perfect residue field.
Then there exist a finite number of prime ideals $\p_1,\dots,\p_n\in\Spec R$ and an integer $m\ge 1$ such that
$$
\DD=\gen{\ess_{\DD}(R/\p_1\oplus\cdots\oplus R/\p_n)}{m}{}
$$
for each $\DD\in\{\Db(\mod R),\Db(\Mod R),\Dm(\mod R),\Dm(\Mod R)\}$.
In particular, one has $\Db(\mod R)=\gen{R/\p_1\oplus\cdots\oplus R/\p_n}{m}{}$ and hence $\dim\Db(\mod R)<\infty$.
\end{theorem}

\begin{proof}
Let us prove the theorem only for $\Db(\mod R)$; the other derived categories can be handled similarly.
We use induction on the Krull dimension $d:=\dim R$. 
If $d=0$, then $R$ is an Artinian ring, and the assertion follows from \cite[Proposition 7.37]{R2}. 
Assume $d\geq1$. 
By \cite[Theorem 6.4]{Ma}, we have a sequence
\[0=I_0\subseteq I_1 \subseteq\cdots\subseteq I_n=R\]
of ideals of $R$ such that for each $1\leq i\leq n$ one has $I_i/I_{i-1}\cong R/\mathfrak{p}_i$ with $\p_i\in\Spec R$.
Then every object $X$ of $\Db(\mod R)$ possesses a sequence
\[
0=XI_0\subseteq XI_1\subseteq\cdots\subseteq XI_n=X
\]
of $R$-subcomplexes.
Decompose this into exact triangles
$$
XI_{i-1} \to XI_i \to XI_i/XI_{i-1} \to \sus XI_{i-1},
$$
in $\Db(\mod R)$, and note that each $XI_i/XI_{i-1}$ belongs to $\Db(\mod R/\p_i)$.
Hence one may assume that $R$ is an integral domain. 
By \cite[Definition-Proposition (1.20)]{Y}, we can take a formal power series subring $A=k[[x_1,\dots,x_d]]$ of $R$ such that $R$ is a finitely generated $A$-module and that the extension $Q(A)\subseteq Q(R)$ of the quotient fields is finite and separable. 

\begin{claim}\label{as poly}
We have natural isomorphisms
\begin{align*}
R & \cong k[[x]][t]/(f(x,t))=k[[x,t]]/(f(x,t)),\\
S:=R\tensor_AR & \cong k[[x]][t,t']/(f(x,t),f(x,t'))=k[[x,t,t']]/(f(x,t),f(x,t')),\\
U:=R\ctensor_k A & \cong k[[x,t,x']]/(f(x,t)),\\
T:=R\ctensor_k R & \cong k[[x,t,x',t']]/(f(x,t),f(x',t')).
\end{align*}
Here $x=x_1,\dots,x_d$, $x'=x'_1,\dots,x'_d$, $t=t_1,\dots,t_n$, $t'=t'_1,\dots,t'_n$ are indeterminates over $k$, and $f(x,t)=f_1(x,t),\dots,f_m(x,t)$ are elements of $k[[x]][t]\subseteq k[[x,t]]$.
In particular, the rings $S,T,U$ are Noetherian commutative complete local rings.
\end{claim}

\begin{proof}[Proof of Claim \ref{as poly}]
Since $R$ is module-finite over $A$, we have $R=Ar_1+\cdots+Ar_n=A[r_1,\dots,r_n]$ for some elements $r_1,\dots,r_n\in R$.
There is an isomorphism
\begin{align*}
R & \cong A[t_1,\dots,t_n]/(f_1(x_1,\dots,x_d,t_1,\dots,t_n),\dots,f_m(x_1,\dots,x_d,t_1,\dots,t_n))\\
& =A[t]/(f(x,t)),
\end{align*}
where $t=t_1,\dots,t_n$ are indeterminates over $A$ and $f(x,t)=f_1(x,t),\dots,f_m(x,t)$ are elements of $A[t]$.
Since $A\subseteq R$ is an integral extension, each element $r_i\in R$ is integral over $A$: there exists an equality $r_i^{\ell_i}+a_{i1}r_i^{\ell_i-1}+\cdots+a_{i\ell_i}=0$ in $R$ with $a_{i1},\dots,a_{i\ell_i}\in A$. 
This implies that the element $t_i^{\ell_i}+a_{i1}t_i^{\ell_{i}-1}+\cdots+a_{i\ell_i}$ belongs to the ideal $(f(x,t))$ of $A[t]$. 
Hence we have
$$
R \cong A[t]/(f(x,t)) =A[[t]]/(f(x,t))=k[[x,t]]/(f(x,t)).
$$
Therefore
\begin{align*}
S &=A[t]/(f(x,t))\tensor_A A[t]/(f(x,t)) \cong A[t,t']/(f(x,t),f(x,t')) \\
& =A[[t,t']]/(f(x,t), f(x,t')) =k[[x,t,t']]/(f(x,t), f(x,t')),
\end{align*}
where $t'=t'_1,\dots,t'_n$ are indeterminates over $A$ which correspond to $t=t_1,\dots,t_n$.
We also have
$$
U \cong k[[x,t]]/(f(x,t))\ctensor_k k[[x]] \cong k[[x,t,x']]/(f(x,t)),
$$
where $x'=x'_1,\dots,x'_d$ are indeterminates over $k$ corresponding to $x=x_1,\dots,x_d$, and
$$
T \cong k[[x,t]]/(f(x,t))\ctensor_k k[[x,t]]/(f(x,t)) \cong k[[x,t,x',t']]/(f(x,t),f(x',t')).
$$
The proof of the claim is completed.
\renewcommand{\qedsymbol}{$\square$}
\end{proof}

There is a surjective ring homomorphism $\mu:S=R\tensor_AR\to R$ which sends $r\tensor r'$ to $rr'$. 
This makes $R$ an $S$-module.
Using Claim \ref{as poly}, we observe that $\mu$ corresponds to the map $k[[x,t,t']]/(f(x,t),f(x,t'))\to k[[x,t]]/(f(x,t))$ given by $t'\mapsto t$.
Taking the kernel, we have an exact sequence
\[0\to I\to S \xrightarrow{\mu}R\to 0\]
of finitely generated $S$-modules.
Along the injective ring homomorphism $A\to S$ sending $a\in A$ to $a\tensor 1=1\tensor a\in S$, we can regard $A$ as a subring of $S$.
Note that $S$ is a finitely generated $A$-module. 
Put $W=A\setminus\{0\}$. 
This is a multiplicatively closed subset of $A$, $R$ and $S$, and one can take localization $(-)_W$.

\begin{claim}\label{our claim}
The $S_W$-module $R_W$ is projective. 
\end{claim}

\begin{proof}[Proof of Claim \ref{our claim}]
The localization $R_W$ is an integral domain which is a module-finite extension of the field $A_W=Q(A)$.
This implies that $R_W$ is also a field (cf. \cite[Lemma 1 in \S 9]{Ma}), and we have $R_W=Q(R)$.
There are natural isomorphisms $S_W\cong(R\tensor_AR)_W\cong R_W\tensor_{A_W}R_W$.
As $Q(R)$ is a separable extension of $Q(A)$, the claim follows. 
\renewcommand{\qedsymbol}{$\square$}
\end{proof}

There are ring epimorphisms
\begin{alignat*}{2}
\alpha & :U\to R, & \quad & r\ctensor a \mapsto ra,\\
\beta & :T\to S, & & r\ctensor r' \mapsto r\tensor r',\\
\gamma & :T\to R, & & r\ctensor r' \mapsto rr'.
\end{alignat*}
Identifying the rings $R$, $S$, $T$ and $U$ with the corresponding residue rings of formal power series rings made in Claim \ref{as poly}, we see that $\alpha,\beta$ are the maps given by $x'\mapsto x$, and $\gamma$ is the map given by $x'\mapsto x$ and $t'\mapsto t$.
Note that $\gamma=\mu\beta$.
The map $\alpha$ is naturally a homomorphism of $(R,A)$-bimodules, and $\beta, \gamma$ are naturally homomorphisms of $(R,R)$-bimodules.
The ring $R$ has the structure of a finitely generated $U$-module through $\alpha$.
The Koszul complex on the $U$-regular sequence $x'-x$ gives a free resolution of the $U$-module $R$:
\begin{equation}\label{gasaf}
0 \to U \to U^{\oplus d} \to U^{\oplus\binom{d}{2}} \to \cdots \to U^{\oplus\binom{d}{2}} \to U^{\oplus d} \xrightarrow{x'-x} U\xrightarrow{\alpha} R\to 0.
\end{equation}
This is an exact sequence of $(R,A)$-bimodules. 
Since the natural homomorphisms
\begin{align*}
& A\cong k[[x']]\to k[[x']][x,t]/(f(x,t)),\\
& k[[x']][x,t]/(f(x,t))\to k[[x']][[x,t]]/(f(x,t))\cong U
\end{align*}
are flat, so is the composition. 
Therefore $U$ is flat as a right $A$-module. 
The exact sequence \eqref{gasaf} gives rise to a chain map $\eta:F\to R$ of $U$-complexes, where
$$
F=(0 \to U \to U^{\oplus d} \to U^{\oplus\binom{d}{2}} \to \cdots \to U^{\oplus\binom{d}{2}} \to U^{\oplus d} \xrightarrow{x'-x} U \to 0)
$$
is a complex of finitely generated free $U$-modules.
By Claim \ref{as poly}, we have isomorphisms
\begin{align*}
U\tensor_AR 
& \cong U\tensor_AA[t]/(f(x,t)) \cong U[t']/(f(x',t'))\cong U[[t']]/(f(x',t'))\\
& \cong (k[[x,t,x']]/(f(x,t)))[[t']]/(f(x',t'))\cong k[[x,t,x',t']]/(f(x,t), f(x',t')) \cong T.
\end{align*}
Note from \cite[Exercise 10.6.2]{W} that $R\tensor^{\bf L}_{A}R$ is an object of $\Dm(R\text{-}\mathsf{Mod}\text{-}R)=\Dm(\Mod R\tensor_kR)$.
(Here, $R\text{-}\mathsf{Mod}\text{-}R$ denotes the category of $(R,R)$-bimodules, which can be identified with $\Mod R\tensor_kR$.)
There are isomorphisms
\begin{align*}
R\tensor^{\bf L}_{A} R &\cong F\tensor_AR \\
&\cong(0 \to U\tensor_AR \to (U\tensor_AR)^{\oplus d} \to \cdots \to (U\tensor_AR)^{\oplus d} \xrightarrow{x'-x} U\tensor_AR\to 0)\\
&\cong (0 \to T \to T^{\oplus d} \to \cdots \to T^{\oplus d} \xrightarrow{x'-x} T\to 0)=:C
\end{align*}
in $\Dm(\Mod R\tensor_kR)$.
Note that $C$ can be regarded as an object of $\Db(\mod T)$.
Taking the tensor product $\eta\tensor_AR$, one gets a chain map $\lambda:C\to S$ of $T$-complexes.
Thus, one has a commutative diagram
\[\begin{CD}
K @>>> C @>>> R @>{\delta}>> \sus K \\
@VVV @V{\lambda}VV @| @V{\xi}VV \\
I @>>> S @>{\mu}>> R @>{\varepsilon}>> \sus I
\end{CD}\]
of exact triangles in $\Db(\mod{T})$.

\begin{claim}\label{ann}
There exists an element $a\in W$ such that $\delta\cdot (1\ctensor a)=0$ in $\homo{\Db(\mod T)}{R}{\sus K}$.
One can choose it as a non-unit element of $A$, if necessary.
\end{claim}

\begin{proof}[Proof of Claim \ref{ann}]
(1) By \cite[Theorem 10.7.4]{W} and Claim \ref{our claim}, we have
\begin{align*}
\varepsilon\tensor_AA_W & \in \homo{\Db(\mod S\tensor_AA_W)}{R\tensor_AA_W}{\sus I\tensor_AA_W} \\
& \cong\ext{S\tensor_AA_W}{1}{R\tensor_AA_W}{I\tensor_AA_W}=0,
\end{align*}
whence $\varepsilon\tensor_AA_W=0$ in $\Db(\mod S\tensor_AA_W)$.
Sending this equality by the exact functor $\Db(\mod S\tensor_AA_W)\to \Db(\mod T\tensor_AA_W)$ induced by the ring epimorphism $\beta\tensor_AA_W:T\tensor_AA_W\to S\tensor_AA_W$, we have $\varepsilon\tensor_AA_W=0$ in $\Db(\mod T\tensor_AA_W)$. 

(2) We have isomorphisms in $\Dm(\Mod R\tensor_kA_W)$
\begin{align*}
& C\tensor_AA_W \cong F\tensor_AR\tensor_AA_W \cong F_W\tensor_{A_W}R_W,\\
& S\tensor_AA_W = R\tensor_AR\tensor_AA_W \cong R_W\tensor_{A_W}R_W,
\end{align*}
which fit into a commutative diagram in $\Dm(\Mod R\tensor_kA_W)$
$$
\begin{CD}
C\tensor_AA_W @>{\lambda\tensor_AA_W}>> S\tensor_AA_W \\
@V{\cong}VV @V{\cong}VV \\
F_W\tensor_{A_W}R_W @>{\eta_W\tensor_{A_W}R_W}>> R_W\tensor_{A_W}R_W.
\end{CD}
$$
The complex $F_W$ is quasi-isomorphic to $R_W$ and $A_W$ is a field.
Hence the morphism $\eta_W\tensor_{A_W}R_W$ above is an isomorphism in $\Dm(\Mod R\tensor_kA_W)$, and so is $\lambda\tensor_AA_W$. 
Since $\lambda\tensor_AA_W$ is a morphism in $\Db(\mod T\tensor_AA_W)$, we have that $\lambda\tensor_AA_W$ is an isomorphism in $\Db(\mod T\tensor_AA_W)$.

(3) Now we have a commutative diagram of exact triangles in $\Db(\mod T\tensor_AA_W)$:
$$
\begin{CD}
K\tensor_AA_W @>>> C\tensor_AA_W @>>> R\tensor_AA_W @>{\delta\tensor_AA_W}>> \sus K\tensor_AA_W \\
@V{\sus^{-1}\xi\tensor_AA_W}VV @V{\lambda\tensor_AA_W}V{\cong}V @| @V{\xi\tensor_AA_W}VV \\
I\tensor_AA_W @>>> S\tensor_AA_W @>{\mu\tensor_AA_W}>> R\tensor_AA_W @>{\varepsilon\tensor_AA_W}>{0}> \sus I\tensor_AA_W.
\end{CD}
$$
We see that $\xi\tensor_AA_W$ is an isomorphism in $\Db(\mod T\tensor_AA_W)$, so that $\delta\tensor_AA_W=0$ in $\Db(\mod T\tensor_AA_W)$.
Let $V$ be the image of $W$ by the ring homomorphism $R\to T$ given by $r\mapsto 1\ctensor r$. 
Clearly, $V$ is a multiplicatively closed subset of $T$, and we have $M\tensor_AA_W\cong M_W=M_V\cong M\tensor_TT_V$ for every $T$-module $M$. 
We get isomorphisms 
\begin{align*}
& \homo{\Db(\mod T\tensor_AA_W)}{R\tensor_AA_W}{\sus K\tensor_AA_W} \\
\cong\ & {\rm H}^1(\RHom_{T\tensor_AA_W}(R\tensor_AA_W, K\tensor_AA_W)) \\
\cong\ & {\rm H}^1(\RHom_{T_V}(R\tensor_TT_V, K\tensor_TT_V)) \\
\cong\ & {\rm H}^1(\RHom_T(R,K))\tensor_TT_V \\
\cong\ & \homo{\Db(\mod T)}{R}{\sus K}\tensor_TT_V \\
\cong\ & \homo{\Db(\mod T)}{R}{\sus K}_V.
\end{align*}
The first isomorphism follows from \cite[Theorem 10.7.4]{W}.
Since $T$ is Noetherian by Claim \ref{as poly} and $R$ is a finitely generated $T$-module, the third isomorphism is obtained by \cite[Exercise 10.8.4]{W}.
Through the above isomorphisms the element $\delta\tensor_AA_W\in\homo{\Db(\mod T\tensor_AA_W)}{R\tensor_AA_W}{\sus K\tensor_AA_W}$ corresponds to the element $\frac{\delta}{1}\in\homo{\Db(\mod T)}{R}{\sus K}_V$. 
As $\delta\tensor_AA_W=0$ in $\Db(\mod T\tensor_AA_W)$, we have $\frac{\delta}{1}=0$. 
Consequently, we find an element $b\in V$ with $\delta\cdot b=0$.
Writing $b=1\ctensor a$ with $a\in W$, we get $\delta\cdot(1\ctensor a)=0$.
If $a$ is a unit of $A$, then $\delta=0$ holds.
Taking any non-unit $a'\in A$ belonging to $W$, we have $\delta\cdot(1\ctensor a')=0$.
\renewcommand{\qedsymbol}{$\square$}
\end{proof}

Let $a\in W$ be a non-unit element of $A$ as in Claim \ref{ann}.
Since we regard $R$ as a $T$-module through the homomorphism $\gamma$, we have an exact sequence $0\to R\xrightarrow{1\ctensor a}R\to R/(a)\to 0$.  
The octahedral axiom makes a diagram in $\Db(\mod T)$ 
\[\begin{CD}
R @>{1\ctensor a}>> R @>>> R/(a) @>>> \sus R \\
@| @V{\delta}VV @VVV @| \\
R @>0>> \sus K @>>> \sus K\oplus \sus R @>>> \sus R \\
@VVV @| @VVV @VVV \\
R @>{\delta}>> \sus K @>>> \sus C @>>> \sus R \\
@VVV @VVV @| @VVV \\
R/(a) @>>> \sus K\oplus \sus R @>>> \sus C @>>> \sus R/(a)
\end{CD}\]
with the bottom row being an exact triangle.
Rotating it, we obtain an exact triangle
\[
K\oplus R\to C\to R/(a)\to \sus (K\oplus R)
\]
in $\Db(\mod T)$.
The exact functor $\Db(\mod T)\to \Dm(\Mod R\tensor_kR)$ induced by the canonical ring homomorphism $R\tensor_kR\to T$ sends this to an exact triangle 
\begin{equation}\label{kii}
K\oplus R\to R\tensor^{\bf L}_{A}R\to R/(a)\to \sus(K\oplus R)
\end{equation}
in $\Dm(\Mod R\tensor_kR)$.
As $R$ is a local domain and $a$ is a non-zero element of the maximal ideal of $R$, we have $\dim R/(a)=d-1<d$. 
Hence one can apply the induction hypothesis to the ring $R/(a)$, and sees that
$$
\Db(\mod R/(a))=\gen{R/\p_1\oplus\cdots\oplus R/\p_h}{m}{R/(a)}
$$
for some integer $m\ge1$ and some prime ideals $\p_1,\cdots,\p_h$ of $R$ that contain $a$.
Now, let $X$ be any object of $\Db(\mod R)$.
Applying the exact functor $X\tensor^{\bf L}_R-$ to \eqref{kii} gives an exact triangle in $\Dm(\Mod R)$
\begin{equation}\label{main triangle}
(X\tensor^{\bf L}_{R}K)\oplus X \to X\tensor^{\bf L}_{A}R \to X\tensor^{\bf L}_{R}R/(a) \to \sus((X\tensor^{\bf L}_{R}K)\oplus X).
\end{equation}
As $A$ has finite global dimension, we observe that $X\tensor^{\bf L}_{A}R$ is homologically bounded.
Since $X\tensor^{\bf L}_{R}R/(a)$ is isomorphic to $X\tensor_RE$, where $E=(0\to R\xrightarrow{a}R\to0)$ is a Koszul complex, $X\tensor^{\bf L}_{R}R/(a)$ is also homologically bounded.
It is clear that the homologies of $X\tensor^{\bf L}_{A}R$ and $X\tensor^{\bf L}_{R}R/(a)$ are finitely generated right $R$-modules.
Hence $X\tensor^{\bf L}_{A}R$ and $X\tensor^{\bf L}_{R}R/(a)$ belong to $\Db(\mod R)$  (cf. \cite[Exercise 10.4.6]{W}), and therefore so does $(X\tensor^{\bf L}_{R}K)\oplus X$ by \eqref{main triangle}.
Thus \eqref{main triangle} is an exact triangle in $\Db(\mod R)$. 
Corollary \ref{krskssn}(2) implies that $X$, as a complex of right $A$-modules, belongs to $\gen{A}{d+1}{A}$, whence $X\tensor^{\bf L}_{A}R$ is in $\gen{R}{d+1}{R}$.
Note that $X\tensor^{\bf L}_{R}R/(a)$ is an object of $\Db(\mod R/(a))=\gen{R/\p_1\oplus\cdots\oplus R/\p_h}{m}{R/(a)}$.
As an object of $\Db(\mod R)$, the complex $X\tensor^{\bf L}_{R}R/(a)$ belongs to $\gen{R/\p_1\oplus\cdots\oplus R/\p_h}{m}{R}$. 
We observe from \eqref{main triangle} that $X$ is in $\gen{R\oplus R/\p_1\oplus\cdots\oplus R/\p_h}{d+1+m}{R}$.
Thus we obtain $\Db(\mod R)=\gen{R\oplus R/\p_1\oplus\cdots\oplus R/\p_h}{d+1+m}{}$.
(As $R$ is a domain, the zero ideal of $R$ is a prime ideal.)
\end{proof}

Here, let us make an easy observation on dimensions of derived categories over localized rings.

\begin{lemma}\label{ess dense of localization}
Let $R$ be a commutative Noetherian ring and $W$ a multiplicatively closed subset of $R$.
For $\star\in\{-,\mathrm{b}\}$ the following hold.
\begin{enumerate}[\rm(1)]
\item
The exact functor
$$
(-)_W=(-)\tensor_R^{\bf L}R_W: \Da(\mod R)\to\Da(\mod R_W)
$$
defined by localization by $W$ is dense.
\item
Suppose that $\Da(\mod R)=\gen{\XX}{n}{}$ for some subcategory $\XX$ of $\Da(\mod R)$ and some integer $n\ge0$.
Then it holds that $\Da(\mod R_W)=\gen{\XX_W}{n}{}$, where $\XX_W$ denotes the subcategory of $\Da(\mod R_W)$ consisting of the objects $X_W$ with $X\in\XX$.
\item
One has $\dim\Da(\mod R_W)\le\dim\Da(\mod R)$.
\end{enumerate}
\end{lemma}

\begin{proof}
The assertions (2) and (3) follow from (1) and Lemma \ref{essentially dense}.
Let us prove the assertion (1).
Let
$$
Y=(\cdots \xrightarrow{d^{t-3}} Y^{t-2} \xrightarrow{d^{t-2}} Y^{t-1} \xrightarrow{d^{t-1}} Y^t \to 0)
$$
be a complex of finitely generated $R_W$-modules.
Then, for each $i\le t$, we can find a finitely generated $R$-module $X^i$ and an isomorphism $f^i:X^i_W\to Y^i$ of $R_W$-modules.
Since $R$ is Noetherian and $X^i$ is finitely generated, there are isomorphisms
$$
\Hom_{R_W}(Y^i,Y^{i+1}) \cong \Hom_{R_W}(X^i_W,X^{i+1}_W) \cong\Hom_R(X^i,X^{i+1})_W,
$$
which correspond $d^i\in\Hom_{R_W}(Y^i,Y^{i+1})$ to $\frac{e^i}{w_i}\in\Hom_R(X^i,X^{i+1})_W$ for some $e^i\in\Hom_R(X^i,X^{i+1})$ and $w_i\in W$.
We have $d^i(f^i(\frac{x}{w}))=f^{i+1}(\frac{e^i(x)}{w_iw})$ for $\frac{x}{w}\in X^i_W$.
The equality $d^{i+1}d^i=0$ in $\Hom_{R_W}(Y^i,Y^{i+2})$ shows that there is an element $v_i\in W$ satisfying $v_ie^{i+1}e^i=0$ in $\Hom_R(X^i,X^{i+2})$.
Setting $\partial^i=v_ie^i$, we have $\partial^{i+1}\partial^i=0$.
Hence the sequence
$$
X:=(\cdots \xrightarrow{\partial^{t-3}} X^{t-2} \xrightarrow{\partial^{t-2}} X^{t-1} \xrightarrow{\partial^{t-1}} X^t \to 0)
$$
of homomorphisms of $R$-modules is a complex.
Put $g^i=v_{t-1}w_{t-1}v_{t-2}w_{t-2}\cdots v_iw_if^i$.
We observe that the maps $g^i$ form an isomorphism $g:X_W\to Y$ of $R_W$-complexes.
\end{proof}

Now, we make sure that the proof of Theorem \ref{main result} also gives a ring-theoretic proof of the affine case of Rouquier's theorem.
Actually, we obtain a more detailed result as follows.
Recall that a commutative ring $R$ is said to be {\em essentially of finite type} over a field $k$ if $R$ is a localization of a finitely generated $k$-algebra.
Of course, every finitely generated $k$-algebra is essentially of finite type over $k$.

\begin{theorem}\label{rqrr}
\begin{enumerate}[\rm (1)]
\item
Let $R$ be a finitely generated algebra over a perfect field.
Then there exist a finite number of prime ideals $\p_1,\dots,\p_n\in\Spec R$ and an integer $m\ge 1$ such that
$$
\DD=\gen{\ess_{\DD}(R/\p_1\oplus\cdots\oplus R/\p_n)}{m}{}
$$
for each $\DD\in\{\Db(\mod R),\Db(\Mod R),\Dm(\mod R),\Dm(\Mod R)\}$.
\item
Let $R$ be a commutative ring which is essentially of finite type over a perfect field.
Then there exist a finite number of prime ideals $\p_1,\dots,\p_n\in\Spec R$ and an integer $m\ge 1$ such that
$$
\DD=\gen{\ess_{\DD}(R/\p_1\oplus\cdots\oplus R/\p_n)}{m}{}
$$
for each $\DD\in\{\Db(\mod R),\Dm(\mod R)\}$.
In particular, one has $\Db(\mod R)=\gen{R/\p_1\oplus\cdots\oplus R/\p_n}{m}{}$.
\end{enumerate}
\end{theorem}

\begin{proof}
(1) In the proof of Theorem \ref{main result}, make the following replacement.
\begin{alignat*}{2}
[[ & \quad & \rightsquigarrow & \quad [ \\
]] & \quad & \rightsquigarrow & \quad ] \\
\widehat\otimes & \quad & \rightsquigarrow & \quad \otimes
\end{alignat*}
Then the assertion follows.
Indeed, the existence of a corresponding subring $A=k[x_1,\dots,x_d]$ of $R$ is seen by \cite[Corollary 16.18]{E}.
The equality $\dim R/(a)=d-1$ in the case where $a$ is not a unit of $R$ is shown by \cite[Corollary 13.11]{E}.

(2) This follows from (1) and Lemma \ref{ess dense of localization}(2).
\end{proof}

Now the following result due to Rouquier (cf. \cite[Theorem 7.38]{R2}) is immediately recovered by Theorem \ref{rqrr}(2).

\begin{corollary}[Rouquier]\label{rqrc}
Let $R$ be a commutative ring essentially of finite type over a perfect field.
Then the derived category $\Db(\mdl R)$ has finite dimension.
\end{corollary}

\begin{remark}\label{vdb}
In Corollary \ref{rqrc}, the assumption that the base field is perfect can be removed; see \cite[Proposition 5.1.2]{KV}.
We do not know whether we can also remove the perfectness assumption of the residue field in Theorem \ref{main result}.
It seems that the techniques in the proof of \cite[Proposition 5.1.2]{KV} do not directly apply to that case.
\end{remark}


\section{Applications}\label{app}

In this section, we give several applications of our theorem obtained in the previous section.
Throughout this section, let $R$ be a commutative Noetherian ring.
We will apply Theorem \ref{main result} to the singularity category of the ring $R$, which has been defined by Orlov \cite{Or1}.
This is the same as the stable derived category in the sense of Buchweitz \cite{Bu}.
Let us recall its definition together with that of the stable category of Cohen-Macaulay modules.

\begin{definition}
\begin{enumerate}[(1)]
\item
The {\em singularity category} $\Dsg(R)$ of $R$, which is also called the {\em stable derived category} of $R$, is defined to be the Verdier quotient of $\Db(\mod R)$ by the subcategory consisting of all perfect $R$-complexes.
(Recall that a {\em perfect complex} is by definition a bounded complex of finitely generated projective modules.)
\item
Let $R$ be a {\em Gorenstein} ring, that is, $R$ has finite injective dimension as an $R$-module.
The {\em stable category} $\sCM(R)$ {\em of Cohen-Macaulay modules} over $R$ is defined as follows.
The objects are the (maximal) {\em Cohen-Macaulay} $R$-modules, i.e., the finitely generated $R$-modules $M$ with $\Ext_R^i(M,R)=0$ for all $i>0$.
The hom-set $\Hom_{\sCM(R)}(M,N)$ is the residue $R$-module of $\Hom_R(M,N)$ by the $R$-submodule consisting of all $R$-homomorphisms from $M$ to $N$ that factor through some finitely generated projective $R$-modules.
\end{enumerate}
\end{definition}

Both the singularity category and the stable category of Cohen-Macaulay modules are known to be triangulated.
The following remarkable result is shown in \cite[Theorem 4.4.1]{Bu}.

\begin{theorem}[Buchweitz]\label{mikan}
Let $R$ be a Gorenstein ring.
Then there is a triangle equivalence $\sCM(R)\cong\Dsg(R)$.
\end{theorem}

Our theorems imply that singularity categories also have finite dimension:

\begin{corollary}\label{53}
Let $R$ be either a complete local ring containing a field with perfect residue field or a ring essentially of finite type over a perfect field.
Then the dimension of $\Dsg(R)$ is finite.
If $R$ is Gorenstein, then $\sCM(R)$ is of finite dimension. 
\end{corollary}

\begin{proof}
The assertion follows from Theorems \ref{main result}, \ref{rqrr}(2), \ref{mikan} and Lemma \ref{essentially dense}.
\end{proof}

Next, we want to consider finiteness of the dimension of the stable category of Cohen-Macaulay modules over a Gorenstein local ring which is not necessarily complete.
For this purpose, we start by making a general lemma on generation of a triangulated category.

\begin{lemma}\label{dimension and thick}
Let $\TT$ be a triangulated category.
Let $G$ be an object of $\TT$ with $\gen{G}{\infty}{}=\TT$.
If the dimension of $\TT$ is finite, then $\TT=\gen{G}{n}{}$ for some integer $n\ge0$.
\end{lemma}

\begin{proof}
The finiteness of the dimension of $\TT$ implies that there exist an object $T\in\TT$ and an integer $m\ge0$ 
such that $\TT=\gen{T}{m}{}$.
Since $\TT=\gen{G}{\infty}{}$, the object $T$ belongs to $\gen{G}{h}{}$ for some $h\ge 0$.
It is shown by induction on $m$ that $\TT=\gen{G}{hm}{}$ holds.
\end{proof}

We will need to use right approximations in the proof of our next result.
Let us recall the definition.

\begin{definition}
Let $\XX$ be a subcategory of $\mod R$.
\begin{enumerate}[\rm (1)]
\item
Let $\phi:X\to M$ be a homomorphism of finitely generated $R$-modules with $X\in\XX$.
We say that $\phi$ is a {\em right $\XX$-approximation} (of $M$) if the induced homomorphism $\Hom_R(X',\phi):\Hom_R(X',X)\to\Hom_R(X',M)$ is surjective for any $X'\in\XX$.
\item
We say that $\XX$ is {\em contravariantly finite} if every finitely generated $R$-module has a right $\XX$-approximation.
\end{enumerate}
\end{definition}

We give a remark on existence of right approximations.

\begin{remark}\label{ks}
Let $X$ be a finitely generated $R$-module.
Then the subcategory $\add_RX$ of $\mod R$ is contravariantly finite.
Indeed, let $M$ be a finitely generated $R$-module.
Since $\Hom_R(X,M)$ is a finitely generated $R$-module, we can take a finite number of elements $f_1,\dots,f_n$ which generate the $R$-module $\Hom_R(X,M)$.
Then it is easily checked that the map
$$
(f_1,\dots,f_n):X^{\oplus n}\to M
$$
is a right $\add_RX$-approximation of $M$.
We should remark here that this does not require uniqueness of indecomposable direct sum decompositions.
\end{remark}

The following lemma will be used in the proof of our next result to descend from the completion of the base local ring.

\begin{lemma}\label{claim2}
Let $R\to S$ be a faithfully flat homomorphism of commutative Noetherian rings.
Let $X,M$ be finitely generated $R$-modules. 
If $M\tensor_RS$ belongs to $\add_S(X\tensor_RS)$, then $M$ belongs to $\add_RX$.
\end{lemma}

\begin{proof}
By Remark \ref{ks} one can take a right $\add_RX$-approximation $\phi: X'\to M$ of $M$.
Let $f:Y\to M\tensor_RS$ be a homomorphism of $S$-modules with $Y\in\add_S(X\tensor_RS)$.
Then there is a pair of $S$-homomorphisms $\theta:Y\to (X\tensor_RS)^{\oplus n}$ and $\pi:(X\tensor_RS)^{\oplus n}\to Y$ 
satisfying $\pi\theta=1$.
The composite map $f\pi$ belongs to $\Hom_S((X\tensor_RS)^{\oplus n},M\tensor_RS)$, which is isomorphic to $\Hom_R(X^{\oplus n},M)\tensor_RS$ since $S$ is faithfully flat over $R$.
Hence we have $f\pi=\sum_{i=1}^mg_i\tensor s_i$ for some $g_i\in\Hom_R(X^{\oplus n},M)$ and $s_i\in S$.
As $\phi$ is a right $\add_RX$-approximation, each $g_i$ factors through $\phi$, i.e., there is a homomorphism $h_i:X^{\oplus n}\to X'$ of $R$-modules such that $g_i=\phi h_i$.
Therefore we have $f=(\phi\tensor_RS)\cdot(\sum_{i=1}^mh_i\tensor s_i)\theta$, namely, $f$ factors through $\phi\tensor_RS$.
Thus, the homomorphism $\phi\tensor_RS: X'\tensor_RS\to M\tensor_RS$ is a right $\add_S(X\tensor_RS)$-approximation of $M\tensor_RS$.

Since $M\tensor_RS$ belongs to $\add_S(X\tensor_RS)$, the identity map of $M\tensor_RS$ factors through $\phi\tensor_RS$.
This means that $\phi\tensor_RS$ is a split epimorphism.
In particular, $\phi\tensor_RS$ is surjective, and hence so is $\phi$ by the faithful flatness of $R\to S$.
The map $\phi$ fits into a short exact sequence
$$
\sigma:0\to K\to X'\xrightarrow{\phi} M\to 0
$$
of $R$-modules, which can be regarded as an element of $\Ext_R^1(M,K)$.
As $\phi\tensor_RS$ is a split epimorphism, the exact sequence
$$
\sigma\tensor_RS:0\to K\tensor_RS\to X'\tensor_RS\xrightarrow{\phi\tensor_RS} M\tensor_RS\to0
$$
splits.
(Exactness follows from the assumption that $S$ is faithfully flat over $R$.)
Hence $\sigma\tensor_RS=0$ in $\Ext_S^1(M\tensor_RS,K\tensor_RS)$.
There are natural maps
$$
\Ext_R^1(M,K)\xrightarrow{\alpha}\Ext_R^1(M,K)\tensor_RS\xrightarrow{\beta}\Ext_S^1(M\tensor_RS,K\tensor_RS)
$$
where $\alpha$ is a monomorphism and $\beta$ is an isomorphism, by the faithful flatness of $S$ over $R$.
The composition $\beta\alpha$ sends $\sigma$ to $\sigma\tensor_RS=0$, hence $\sigma=0$.
This means that the exact sequence $\sigma$ splits, and $\phi$ is a split epimorphism.
Therefore $M$ is isomorphic to a direct summand of $X'\in\add_RX$, and we conclude that $M$ is in $\add_RX$.
\end{proof}

Now we can prove the following result, which extends the second assertion of Corollary \ref{53} to the non-complete case.
For a finitely generated $R$-module and an integer $n\ge0$, the $n$-th syzygy of $M$ is denoted by $\mathrm{\Omega}^nM$.

\begin{theorem}
Let $R$ be a $d$-dimensional Gorenstein local ring containing a field with perfect residue field $k$. 
Assume that $R$ is excellent and has an isolated singularity. 
Then one has $\sCM(R)=\gen{\mathrm{\Omega}^dk}{n}{}$ for some positive integer $n$. 
In particular, the category $\sCM(R)$ has finite dimension.
\end{theorem}

\begin{proof}
Put $G=\mathrm{\Omega}^dk$.
As is well-known, for an excellent local ring the property of having an isolated singularity is stable under completion.
Hence the completion $\widehat{R}$ of $R$ also has an isolated singularity. 
Then it follows from \cite[Corollary 2.9]{stcm} that $\sCM(\widehat{R})=\gen{\widehat{G}}{\infty}{}$ holds. 
Corollary \ref{53} and Lemma \ref{dimension and thick} yield $\sCM(\widehat{R})=\gen{\widehat{G}}{n}{}$ for some $n\ge 1$. 

Here we establish a claim.

\begin{claim*}\label{claim1}
Let $m$ be a positive integer.
For any $N\in\gen{\widehat{G}}{m}{}$ there exists $M\in\gen{G}{m}{}$ such that $N$ is isomorphic to a direct summand of $\widehat{M}$ in $\sCM(\widehat{R})$. 
\end{claim*}

\begin{proof}[Proof of Claim]
We show this by induction on $m$.
If $m=1$, then in $\sCM(R)$ the module $N$ is isomorphic to a direct summand of a direct sum $\bigoplus_{i=1}^{h}\sus^{l_i}\widehat{G}$, and one can take $M=\bigoplus_{i=1}^{h}\sus^{l_i}G$. 
Assume $m\geq 2$.
There is an exact triangle
$$
X\xrightarrow{f}Y\to Z\to \sus X
$$
in $\sCM(\widehat{R})$ with $X\in\gen{\widehat{G}}{m-1}{}$ and $Y\in\gen{\widehat{G}}{}{}$ such that $N$ is a direct summand of $Z$ . 
The induction hypothesis implies that there exist objects $V\in\gen{G}{m-1}{}$ and $W\in\gen{G}{}{}$ such that $X$ and $Y$ are isomorphic to direct summands of $\widehat{V}$ and $\widehat{W}$, respectively.
Since the exact functor
$$
\widehat{(-)}=(-)\tensor_R\widehat R: \sCM(R)\to \sCM(\widehat{R})
$$
is fully faithful (cf. \cite[Proposition 1.5]{KMV}), we obtain a morphism $g:V\to W$ in $\sCM(R)$ and a commutative diagram
$$
\begin{CD}
X @>f>> Y \\
@V{\alpha}VV @V{\beta}VV \\
\widehat{V} @>{\widehat{g}}>> \widehat{W}
\end{CD}
$$
in $\sCM(\widehat{R})$ such that $\alpha,\beta$ are split monomorphisms.
Taking the mapping cone of $g$ gives an exact triangle
$$
V\xrightarrow{g}W\to M\to\sus V
$$
in $\sCM(R)$.
Since $V\in\gen{G}{m-1}{}$ and $W\in\gen{G}{}{}$, the module $M$ belongs to $\gen{G}{m}{}$.
We have a commutative diagram of exact triangles in $\sCM(\widehat{R})$
\[\begin{CD}
X @>f>> Y @>>> Z @>>> \sus X \\
@V{\alpha}VV @V{\beta}VV @V{\gamma}VV @VVV \\
\widehat{V} @>{\widehat{g}}>> \widehat{W} @>>> \widehat{M} @>>> \sus \widehat{V}
\end{CD}\]
and we can easily verify that $\gamma$ is a split monomorphism.
Thus in $\sCM(R)$ the object $N$ is isomorphic to a direct summand of $\widehat{M}$, which belongs to $\gen{\widehat{G}}{m}{}$.
\renewcommand{\qedsymbol}{$\square$}
\end{proof}

Let $X$ be an object of $\sCM(R)$.
Then the completion $\widehat X$ is in $\sCM(\widehat R)=\gen{\widehat{G}}{n}{}$.
The above claim implies that there exists an object $Y\in\gen{G}{n}{}$ such that $\widehat{X}$ is isomorphic to a direct summand of $\widehat{Y}$ in $\sCM(\widehat{R})$.
Hence the module $\widehat X$ belongs to $\add_{\widehat R}(\widehat Y\oplus\widehat R)=\add_{\widehat R}((Y\oplus R)\tensor_R\widehat R)$.
It is observed by Lemma \ref{claim2} that $X$ is in $\add_R(Y\oplus R)$, whence $X$ belongs to $\gen{G}{n}{}$ in $\sCM(R)$.
\end{proof}


\section{Lower bounds}\label{lower}

In this section, we will mainly study lower bounds for the dimension of the bounded derived category of finitely generated modules.
We shall refine a result of Rouquier over an affine algebra, and also give a similar lower bound over a general commutative Noetherian ring.
We will end this section by mentioning lower bounds for dimensions of singularity categories and stable categories of Cohen-Macaulay modules.

Throughout this section, let $R$ be a commutative Noetherian ring.
First, we consider refining a result of Rouquier.
We start by stating a lemma obtained from basic ideal theory.
We denote the set of maximal ideals of $R$ by $\Max R$.
When $R$ has finite Krull dimension, $\Assh R$ denotes the set of prime ideals $\p$ of $R$ with $\dim R/\p=\dim R$.

\begin{lemma}\label{bit}
Let $R$ be a finitely generated algebra over a field.
Then the equality
$$
\bigcap_{\p\in\Assh R}\p=\bigcap_{
\begin{smallmatrix}
\m\in\Max R\\
\height\m=\dim R
\end{smallmatrix}
}\m
$$
holds.
\end{lemma}

\begin{proof}
Let $d=\dim R$.
Take $\p\in\Assh R$.
Then $\p=\sqrt\p$ coincides with the intersection of the maximal ideals of $R$ containing $\p$; see \cite[Theorem 5.5]{Ma}.
Hence $\bigcap_{\p\in\Assh R}\p$ is equal to the intersection of the maximal ideals containing some prime ideals in $\Assh R$.
Thus it is enough to show that a maximal ideal of $R$ has height $d$ if and only if it contains some prime ideal in $\Assh R$.
Let $\m$ be a maximal ideal of $R$.
If $\m$ is of height $d$, then there is a chain $\p_0\subsetneq\p_1\subsetneq\cdots\subsetneq\p_d=\m$ in $\Spec R$, and $\p_0$ belongs to $\Assh R$.
Conversely, suppose that $\m$ contains $\p\in\Assh R$.
Then $R/\p$ is an integral domain that is finitely generated over a field.
By \cite[Corollary 13.4]{E} we have $\dim R/\p=\height\m/\p$.
Therefore $\height\m=d$ holds.
\end{proof}

For an $R$-module $M$, its {\em non-free locus} $\operatorname{NF}(M)$ is defined to be the set of prime ideals $\p$ of $R$ such that the $R_\p$-module $M_\p$ is non-free.
It is a well-known fact that if $M$ is finitely generated, then $\operatorname{NF}(M)$ is a closed subset of $\Spec R$ in the Zariski topology; see \cite[Corollary 2.11]{res} for instance.

Now, under a mild assumption, we can prove that for an affine algebra $R$, the Krull dimension of $R$ is a lower bound for the dimension of the derived category $\Db(\mod R)$.

\begin{theorem}\label{dimdim}
Let $R$ be a finitely generated algebra over a field.
Suppose that there exists $\p\in\Assh R$ such that $R_\p$ is a field.
Then one has the inequality $\dim\Db(\mod R)\ge\dim R$.
\end{theorem}

\begin{proof}
The theorem is obvious when the left hand side is infinite, so assume that $\dim\Db(\mod R)=:n$ is finite.
We can write $\Db(\mod R)=\gen{G}{n+1}{}$ for some object $G\in\Db(\mod R)$.
Set $M=\bigoplus_{i\in\Z}{\rm H}^iG$.
This is a finitely generated $R$-module, and one can write $\operatorname{NF}(M)=V(I)$ for some ideal $I$ of $R$.
Since $R_\p$ is a field, $\p$ does not belong to $\operatorname{NF}(M)$, equivalently, $\p$ does not contain $I$.
As $\p$ belongs to $\Assh R$, we have
$$
I\nsubseteq\bigcap_{\q\in\Assh R}\q=\bigcap_{
\begin{smallmatrix}
\m\in\Max R\\
\height\m=\dim R
\end{smallmatrix}
}\m
$$
by Lemma \ref{bit}.
Therefore we can choose a maximal ideal $\m$ with $\height\m=\dim R$ that does not contain the ideal $I$.
Then $\m$ is not in $V(I)=\operatorname{NF}(M)$, so that the $R_\m$-module $M_\m$ is free.
For each integer $i$, the module ${\rm H}^i(G_\m)$ is isomorphic to $({\rm H}^iG)_\m$, which is a direct summand of the free $R_\m$-module $M_\m$.
Hence ${\rm H}^i(G_\m)$ is also a free $R_\m$-module.
Lemma \ref{split out} implies that in $\Db(\mod R_\m)$ the localized complex $G_\m$ is isomorphic to 
$\bigoplus_{i\in\Z}\sus^{-i}{\rm H}^i(G_\m)$, which belongs to $\gen{R_\m}{}{}$.
By Lemma \ref{ess dense of localization}(2), we have $\Db(\mod R_\m)=\gen{G_\m}{n+1}{}=\gen{R_\m}{n+1}{}$.
In particular, the residue field $\kappa(\m)$ belongs to $\gen{R_\m}{n+1}{}$, and it follows from \cite[Lemma 4.3]{KK} that $\dim R=\height\m=\dim R_\m<n+1$.
\end{proof}

The following result of Rouquier \cite[Proposition 7.16]{R2} is a direct consequence of Theorem \ref{dimdim}.

\begin{corollary}[Rouquier]
Let $R$ be a reduced finitely generated algebra over a field.
Then $\dim\Db(\mod R)\ge\dim R$.
\end{corollary}

Next, we try to extend Theorem \ref{dimdim} to non-affine algebras.
We do not know whether the inequality in Theorem \ref{dimdim} itself holds over non-affine algebras; we can prove that a similar but slightly weaker inequality holds over them.
We need the lemma below, which is also shown by basic ideal theory.
For an ideal $I$ of $R$, let $\Min_RR/I$ denote the set of minimal prime ideals of $I$.

\begin{lemma}\label{codi}
Let $R$ be a local ring of Krull dimension $d\ge 1$.
Let
$$
I=\bigcap_{
\begin{smallmatrix}
\p\in\Spec R\\
\height\p=d-1
\end{smallmatrix}
}\p
$$
be an ideal of $R$.
Then one has $\dim R/I=d$.
\end{lemma}

\begin{proof}
We use induction on $d$.
When $d=1$, the ideal $I$ coincides with the nilradical of $R$, and hence $\dim R/I=d$ holds.
Let $d\ge 2$, and denote by $\m$ the maximal ideal of $R$.
There exists a chain
$$
\p_0\subsetneq\p_1\subsetneq\cdots\subsetneq\p_{d-1}\subsetneq\p_d=\m
$$
of prime ideals of $R$.
Then since $\p_{d-1}$ has height $d-1$, we have inclusions $I\subseteq\p_{d-1}\subsetneq\m$.
This especially says that $\m$ is not in $\Min_RR/I$.
As $d>0$, we have that $\m$ does not belong to the set $\Min R\cup\Min_RR/I$.
By prime avoidance, one finds an element $x\in\m$ which is not in any prime ideal in $\Min R\cup\Min_RR/I$.
We have $\dim R/(x)=d-1$ and $\dim R/I+(x)=\dim R/I-1$.

Let $P$ be a prime ideal of the residue ring $R/(x)$ of height $d-2$.
Write $P=\p/(x)$ for some $\p\in V(x)$.
Note that no minimal prime ideal of $R_\p$ contains $\frac{x}{1}$.
We have equalities
$$
d-2=\height P=\height\p/(x)=\dim R_\p/(\textstyle\frac{x}{1})=\dim R_\p-1=\height\p-1.
$$
Hence $\p$ is of height $d-1$.
Therefore $I$ is contained in $\p$, and $I(R/(x))$ is contained in $P$.
Thus, the ideal $I(R/(x))$ of $R/(x)$ is contained in
$$
I':=\bigcap_{
\begin{smallmatrix}
P\in\Spec R/(x)\\
\height P=d-2
\end{smallmatrix}
}P.
$$
The induction hypothesis gives an equality $\dim (R/(x))/I'=\dim R/(x)$, which implies
$$
\dim R/I-1=\dim R/I+(x)=\dim (R/(x))/I(R/(x))=\dim R/(x)=d-1.
$$
The equality $\dim R/I=d$ follows.
\end{proof}

Now we can show the following result.

\begin{theorem}\label{dimdim-1}
Let $R$ be a ring of finite Krull dimension such that $R_\mathfrak{p}$ is a field for all $\p\in\Assh R$. 
Then we have $\dim \Db(\mod R)\geq \dim R-1$. 
\end{theorem}

\begin{proof}
The assertion trivially holds if either $\dim\Db(\mod R)=\infty$ or $\dim R=0$.
Let $n:=\dim\Db(\mod R)<\infty$ and $d:=\dim R\ge 1$.
Then we have $\Db(\mod R)=\gen{G}{n+1}{}$ for some object $G\in\Db(\mod R)$.
Put $M=\bigoplus_{i\in\Z}{\rm H}^iG$.
As $M$ is finitely generated, we have $\operatorname{NF}(M)=V(I)$ for some ideal $I$ of $R$.
The argument in the latter half of the proof of Theorem \ref{dimdim} (with $\m$ replaced by $\p$) shows that $\height\p<n+1$ for every $\p\in D(I)$.
Thus it suffices to find a prime ideal $\p\in D(I)$ with $\height\p\ge d-1$.
Take a maximal ideal $\m$ of $R$ of height $d$.
If $\m$ does not contain the ideal $I$, then we are done.
So we assume that $\m$ contains $I$.

We claim here that the inequality $\dim R_\m/IR_\m<\dim R_\m$ holds.
In fact, suppose that we have $\dim R_\m/IR_\m=\dim R_\m$.
Then $\height\m/I=\height\m=d$, which implies that there exist prime ideals $\p_0,\dots,\p_d$ of $R$ that satisfy
$$
I\subseteq\p_0\subsetneq\p_1\subsetneq\cdots\subsetneq\p_d=\m.
$$
It follows from this chain that $\p_0$ is a prime ideal in $\Assh R$ containing $I$.
This implies that $\Assh R\cap\operatorname{NF}(M)\ne\varnothing$, contrary to the assumption of the theorem.

This claim and Lemma \ref{codi} imply that $IR_\m$ is not contained in the intersection of the prime ideals of $R_\m$ of height $d-1$.
Hence we find a prime ideal $P\in\Spec R_\m$ with $\height P=d-1$ that does not contain $IR_\m$.
Write $P=\p R_\m$ for some prime ideal $\p$ of $R$ contained in $\m$.
Then $\p$ is in $D(I)$ and $\height\p=d-1$.
\end{proof}

Here is an obvious conclusion of the above theorem.

\begin{corollary}
Let $R$ be a reduced ring of finite Krull dimension.
Then $\dim\Db(\mod R)\ge\dim R-1$.
\end{corollary}

Finally, we give a note on lower bounds for the dimensions of the singularity category and the stable category of Cohen-Macaulay modules.
For a local ring $R$ with maximal ideal $\m$, the {\em codimension} of $R$ is defined by
$$
\codim R=\edim R-\dim R,
$$
where $\edim R$ denotes the embedding dimension of $R$, i.e., $\edim R=\dim_{R/\m}\m/\m^2$.
The following result is proved by Bergh, Iyengar, Krause and Oppermann \cite[Corollary 5.10]{BIKO} in the case where $R$ is $\m$-adically complete, and done by Avramov and Iyengar \cite{AI} in the general case.

\begin{theorem}[Avramov-Iyengar, Bergh-Iyengar-Krause-Oppermann]\label{aibiko}
Let $R$ be a complete intersection local ring.
Then there is an inequality $\dim\Dsg(R)\ge\codim R-1$.
\end{theorem}

By using this theorem, we obtain a lower bound over an arbitrary commutative Noetherian ring.
We denote by $\CI(R)$ the {\em complete intersection locus} of $R$, namely, the set of prime ideals $\p$ of $R$ such that the local ring $R_\p$ is a complete intersection.

\begin{proposition}\label{saigoo}
Let $R$ be a ring.
Then the inequality
$$
\dim\Dsg(R)\ge\sup\{\,\codim R_\p\mid\p\in\CI(R)\,\}-1.
$$
holds.
If $R$ is Gorenstein, then
$$
\dim\sCM(R)\ge\sup\{\,\codim R_\p\mid\p\in\CI(R)\,\}-1.
$$
\end{proposition}

\begin{proof}
The second assertion follows from the first one and Theorem \ref{mikan}.
We may assume that $\dim\Dsg(R)$ is finite, say $n$.
There is an object $G\in\Dsg(R)$ such that $\Dsg(R)=\gen{G}{n+1}{}$.
Fix a prime ideal $\p\in\CI(R)$.
Using Lemma \ref{ess dense of localization}(1), we see that $\Dsg(R_\p)=\gen{G_\p}{n+1}{}$ holds.
Hence we have $n\ge\dim\Dsg(R_\p)$.
Since $R_\p$ is a complete intersection local ring, the right hand side is more than or equal to $\codim R_\p-1$ by Theorem \ref{aibiko}.
\end{proof}

\begin{remark}\label{ai}
Avramov and Iyengar \cite{AI} actually give a lower bound for the dimension of the singularity category of an arbitrary local ring, using its {\em conormal free rank}.
A similar argument to the proof of Proposition \ref{saigoo} can give a better lower bound than the proposition.
\end{remark}

\section*{Acknowledgments}
The authors express their gratitude to Osamu Iyama for his crucial questions and references, to Srikanth Iyengar for his significant suggestions throughout this paper, and to Yuji Yoshino for his valuable discussions with the second author on the proof of Theorem \ref{main result}.
The authors are indebted to Luchezar Avramov and Shunsuke Takagi for their helpful advice on the contents of Section \ref{lower}.
The authors thank Naoyuki Matsuoka and Michel Van den Bergh for their useful comments concerning Corollary \ref{mtk} and Remark \ref{vdb}, respectively.
Part of this work was done during the stay of the second author at University of Nebraska-Lincoln in 2011.
He is grateful for their kind hospitality.


\end{document}